\documentclass[11pt]{article}
\usepackage{amsfonts}
\usepackage{amssymb,amsmath,mathrsfs, amsthm}

\numberwithin{equation}{section}

\newtheorem{theorem}{Theorem}[section]

\newtheorem{proposition}{Proposition}[section]
\theoremstyle{definition}
\newtheorem{definition}{Definition}[section]
\newtheorem{remark}{Remark}[section]

\theoremstyle{remark}

\date{}
\begin{document}

\title{De Rahm decomposition theorem for strongly convex K\"ahler-Berwald manifolds}
\author{Chunping Zhong\\
School of Mathematical Sciences, Xiamen
University, Xiamen 361005, China\\
E-mail: zcp@xmu.edu.cn
}

\date{}
\maketitle

\begin{abstract}
Let $B_n$ and $P_n$ be the unit ball and the unit polydisk in $\mathbb{C}^n$ with $n\geq 2$ respectively. Denote $\mbox{Aut}(B_n)$ and $\mbox{Aut}(P_n)$ the
holomorphic automorphism group of $B_n$ and $P_n$ respectively. In this paper, we prove that $B_n$ admits no $\mbox{Aut}(B_n)$-invariant
strongly pseudoconvex complex Finsler metric other than a constant multiple of the Poincar$\acute{\mbox{e}}$-Bergman metric,
while $P_n$ admits infinite many $\mbox{Aut}(P_n)$-invariant complete strongly convex complex Finsler metrics
other than the Bergman metric.
The $\mbox{Aut}(P_n)$-invariant complex Finsler metrics are explicitly constructed which depend on a real parameter $t\in [0,+\infty)$ and
integer $k\geq 2$. These metrics are proved to be strongly convex K\"ahler-Berwald metrics, and they posses very similar properties as that of the Bergman metric on $P_n$. As applications, the existence of $\mbox{Aut}(M)$-invariant strongly convex complex Finsler metrics is also investigated on some Siegel domains of the first and the second kind which are biholomorphic equivalently to the unit polydisc in $\mathbb{C}^n$. We also give a characterization of strongly convex K\"ahler-Berwald spaces and give a de Rahm type decomposition theorem for strongly convex K\"ahler-Berwald spaces.
\end{abstract}

\section{Introduction and main results}

Let $M$ be a complex manifold and $T^{1,0}M$ its holomorphic tangent bundle. An upper semicontinuous function $F:T^{1,0}M\rightarrow [0,+\infty)$ is called a complex Finsler metric if it satisfies
$$F(p,\lambda v)=|\lambda|F(p,v),\quad \forall p\in M, \forall v\in T_p^{1,0}M,\forall \lambda\in\mathbb{C}.$$
An intrinsic metric is a biholomorphic invariant complex Finsler metric which is determined by the complex analytic structure on $M$.  A complex manifold $M$ endowed with a strongly pseudoconvex complex Finsler metric $F$ is called a complex Finsler manifold, denoted by $(M,F)$.

  Complex Finsler metrics arise naturally in geometric function theory of several complex variables since most well-known intrinsic metrics (such as the Carath$\acute{\mbox{e}}$odory, Kobayashi and Bergman metrics)
on complex manifolds are complex Finsler metrics in nature. These metrics play a very important role in geometric function theory of several complex variables \cite{Ko}.

In \cite{Aikou1}, Aikou introduced the notion of complex Berwald manifold which is a class of special complex Finsler manifolds. Complex Berwald manifolds are important and interesting because every complex Berwald metric is affinely equivalent to a Hermitian metric.

In the same paper Aikou also constructed complex Berwald metrics on $M$ by using a Hermitian metric $\pmb{\alpha}=\sqrt{a_{i\overline{j}}(z)v^i\overline{v^j}}$ and $1$-form $\pmb{\beta}=b_i(z)v^i$ of type $(1,0)$ on $M$ and proved that if the $1$-form $\pmb{\beta}$ is holomorphic and parallel with respect to $\pmb{\alpha}$, then the obtained metric $F(\pmb{\alpha},\pmb{\beta})$ is a complex Berwald metric, and if furthermore $\pmb{\alpha}$ is a K\"ahler metric then $F(\pmb{\alpha},\pmb{\beta})$ is a K\"ahler-Berwald metric.
Furthermore, Aikou proved that for each complex Berwald manifold (resp. K\"ahler-Berwald manifold) $(M,F)$, there exists a Hermitian metric (resp. K\"ahler metric) $h_M$ on $M$ such that the horizontal Chern-Finsler connection (cf. \cite{ap}, also called complex Rund connection in \cite{Aikou1}, or Finsler connection in \cite{Aikou2}) associated to $F$ coincides with the pull-back of the Hermitian connection associated to $h_M$. Therefore a K\"ahler-Berwald manifold is necessary a K\"ahler manifold.

 Aikou's construction of complex Berwald metrics impose some restrictions on $M$, i.e., the existence of holomorphic and parallel $1$-form $\pmb{\beta}$ on $M$. On the other hand, in the study of $U(n)$-invariant strongly pseudoconvex complex Finsler metrics on domains $D\subset\mathbb{C}^n$, the author \cite{Zh2} proved that there actually exists no $U(n)$-invariant complex Berwald metric other than those $U(n)$-invariant Hermitian metrics on $D$.

\textbf{Question 1.} Under what condition does a complex manifold $M$ admit a complex Berwald metric $F$ which is not a Hermitian quadratic metric?
Under what condition does a K\"ahler manifold $M$ admit a K\"ahler-Berwald metric $F$ which is not a Hermitian quadratic metric?

In this paper, we shall construct complex Berwald metrics on reducible Hermitian manifolds. As an important application we show that there exists infinite many $\mbox{Aut}(P_n)$-invariant strongly convex K\"ahler-Berwald metrics on the unit polydisk $P_n$ in $\mathbb{C}^n$ with $n\geq 2$, while there exists no $\mbox{Aut}(B_n)$-invariant strongly pseudoconvex complex Finsler metric other than a constant multiple of the Poincar$\acute{\mbox{e}}$-Bergman metric on the  unit ball $B_n$ in $\mathbb{C}^n$. We are able to give a characterization of strongly convex K\"ahler-Berwald spaces and give a de Rham type decomposition theorem for strongly convex K\"ahler-Berwald spaces.

The following theorem gives an answer to \textbf{Question 1}.

\begin{theorem}\label{mth-1}
Let $(M,Q)$ be a simply connected complete reducible $C^\infty$ Hermitian manifold (resp. K\"ahler manifold) such that $(M_1,Q_1)\times \cdots\times (M_n,Q_n)$ is the de Rham decomposition of $(M,Q)$ and $\pmb{\pi}_l:M\rightarrow M_l$ are the natural projections of $M$ onto the $l$-th factors $M_l$ for $l=1,\cdots,n$. Then

(1) $M$ admits infinite many strongly convex complete complex Berwald metrics (resp. K\"ahler-Berwald metrics)
\begin{equation}
F_{t,k}(z,v)=\frac{1}{\sqrt{1+t}}\sqrt{\sum_{l=1}^nQ_l(\pmb{\pi}_l(z),(\pmb{\pi}_l)_\ast(v))+t\sqrt[k]{\sum_{l=1}^nQ_l^k(\pmb{\pi}_l(z),(\pmb{\pi}_l)_\ast(v))}},\quad\forall (z,v)\in T^{1,0}M,\label{ftk}
\end{equation}
where $t\in[0,+\infty)$ and $k\geq 2$ is an integer;

(2) $F_{t,k}$ is a complete real Berwald metric on $M$, thus $F_{t,k}$ has vanishing $\textbf{S}$-curvature for any $t\in[0,+\infty)$ and integer $k\geq 2$.

\end{theorem}
\begin{remark}
(1) For $t=0$, $F_{0}:=\sqrt{Q_1+\cdots+Q_n}$ is the usual product metric, which is $C^\infty$ over the whole holomorphic tangent bundle $T^{1,0}M$; for each $t\in(0,+\infty)$ and each fixed integer $k\geq 2$, however, $F_{t,k}$ is a non-Hermitian quadratic complex Finsler metric on $M$ which is only $C^\infty$ on $\widetilde{M}=T^{1,0}M\setminus\{\mbox{zero section}\}$. (2) If we consider $t$ as the parameter of time, $F_{t,k}$ can be considered as a deformation of $F_0$, and indeed $F_{t,k}$ posses some properties which are very similar to the usual product metric. (3) For any $t_1,t_2\in[0,+\infty)$ with $t_1\neq t_2$, and integers $k_1,k_2\geq 2$ with $k_1\neq k_2$, the complex Finsler manifolds $(M,F_{t_1,k_1})$ and $(M,F_{t_2,k_2})$ are not holomorphic isometrically equivalent.
\end{remark}

   Comparing with the Bergman metric on a complex manifold $M$, it is well-known that the Carath$\acute{\mbox{e}}$odory and Kobayashi metrics on $M$ are in general not smooth complex Finsler metrics (in the sense of \cite{ap}) and even not Hermitian quadratic. In \cite{Lempert}, Lempert proved a fundamental theorem which states that on bounded strongly convex domains with smooth boundaries in $\mathbb{C}^n$, the Kobayashi and Caratheodory metrics coincide and they are strongly pseudoconvex weakly K\"ahler-Finsler metrics with constant holomorphic sectional curvature $-4$ (cf. \cite{ap}). Since then several important works were carried out to give various characterizations of the Kobayashi metrics on complex manifolds from the viewpoint of complex Finsler geometry (cf. \cite{ap} and the references therein).
    Unless in some very special cases, however, these metrics cannot be explicitly expressed.
  In \cite{Zh2}, the author studied $U(n)$-invariant strongly pseudoconvex complex Finsler metrics on domains in $\mathbb{C}^n$ and gave some characterizations of strongly pseudoconvex $U(n)$-invariant complex Finsler metrics. Our studies show that there is no $U(n)$-invariant non-Hermitian quadratic K\"ahler-Finsler metric on a $U(n)$-invariant domain in $\mathbb{C}^n$. But there are indeed lots of weakly complex Berwald metrics in the sense of \cite{Zh1} which are $U(n)$-invariant strongly pseudoconvex (even strongly convex) and non-Hermitian quadratic.

Usually, in dealing with problems in geometric function theory of several complex variables, it is more desirable for us to construct and use holomorphic invariant metrics.
 Denote $\mbox{Aut}(M)$ the holomorphic automorphism group of a complex manifold $M$. It is known that the Bergman metric on $M$ is $\mbox{Aut}(M)$-invariant and Hermitian quadratic. So far to my knowledge, there was not any explicit example of complex Finsler metric on a complex manifold $M$ in literatures which is simultaneously $\mbox{Aut}(M)$-invariant, strongly pseudoconvex and non-Hermitian quadratic. In order to study complex Finsler geometry and investigate its possible applications in several complex variables, we need more examples of $\mbox{Aut}(M)$-invariant strongly pseudoconvex complex Finsler metrics-both Hermitian quadratic (such as the Bergman metric) and non-Hermitian quadratic (to be found yet!) (cf. \cite{AAP}). Thus it is very natural and interesting to ask the following question.

\textbf{Question 2.} Under what condition does a complex manifold $M$ admits an $\mbox{Aut}(M)$-invariant complex Finsler metric which is strongly pseudoconvex, non-Hermitian quadratic?

If there are such metrics, one may want to know that among such metrics whether there are K\"ahler-Finsler metrics or weakly K\"ahler-Finsler metrics in the sense of \cite{ap}. By Lempert's fundamental results in \cite{Lempert}, the Carath$\acute{\mbox{e}}$odory and Kobayashi metrics on bounded strongly convex domains with smooth boundaries in $\mathbb{C}^n$ are strongly pseudoconvex weakly K\"ahler-Finsler metrics with constant holomorphic sectional curvature $-4$, it is not known yet whether they are K\"ahler-Berwald metrics.

The following theorem shows that the existence of non-Hermitian quadratic $\mbox{Aut}(M)$-invariant strongly pseudoconvex complex Finsler metric on a complex manifold $M$ is not always possible!
\begin{theorem}\label{mth-2}
The unit ball $B_n=\{\|z\|^2<1\}$ in $\mathbb{C}^n$ admits no $\mbox{Aut}(B_n)$-invariant strongly pseudoconvex complex Finsler metrics other than a constant multiple of the Poincar$\acute{\mbox{e}}$-Bergman metric.
\end{theorem}
Actually the above theorem holds for a class of Kobayashi-hyperbolic manifolds.
\begin{theorem}
Let $M$ be a connected Kobayashi-hyperbolic manifold of complex dimension $n$. Suppose that $\dim\mbox{Aut}(M)=n^2+2n$. Then $M$ admits no $\mbox{Aut}(M)$-invariant strongly pseudoconvex complex Finsler metric other than a constant multiple of the Bergman metric, namely $M$ admits no $\mbox{Aut}(M)$-invariant non-Hermitian quadratic strongly pseudoconvex complex Finsler metric.
\end{theorem}

The following theorem shows that on the polydisk $P_n$ in $\mathbb{C}^n$ with $n\geq 2$, however, there are infinite many strongly convex complex Finsler metrics which are $\mbox{Aut}(P_n)$-invariant and non-Hermitian quadratic.
\begin{theorem}\label{mth-3}
Let $P_n$ be the  unit polydisk in $\mathbb{C}^n$ with $n\geq 2$. Then for any fixed $t\in[0,+\infty)$ and integer $k\geq 2$,
\begin{eqnarray*}
F_{t,k}(z,v)=\frac{1}{\sqrt{1+t}}\sqrt{\sum_{l=1}^n\frac{|v^l|^2}{(1-|z^l|^2)^2}+t\sqrt[k]{\sum_{l=1}^n\frac{|v^l|^{2k}}{(1-|z^l|^2)^{2k}}}},\quad \forall (z,v)\in T^{1,0}P_n
\end{eqnarray*}
is an $\mbox{Aut}(P_n)$-invariant complete strongly convex K\"ahler-Berwald metric.
\end{theorem}

\begin{remark}
For $t=0$, $F_0^2=\displaystyle\sum_{l=1}^n\frac{|v^l|^2}{(1-|z^l|^2)^2}$ is the Bergman metric on $P_n$ which is clearly $\mbox{Aut}(P_n)$-invariant, Hermitian quadratic and $C^\infty$ on the whole holomorphic tangent bundle $T^{1,0}P_n$; for any fixed $t\in(0,+\infty)$ and integer $k\geq 2$, however, $F_{t,k}^2$ are non-Hermitian quadratic metrics and are only $C^\infty$ on the complement $\widetilde{P_n}$ of the zero section in $T^{1,0}P_n$, thus they serve as an important class of holomorphic invariant strongly convex complex Finsler metrics in the strict sense of \cite{ap}.
\end{remark}



\begin{remark}\label{remark-3}
Note that by Proposition 5.2 in \cite{Zh4}, for any fixed $t\in(0,+\infty)$, integer $k\geq 2$ and $j\in\{1,2,\cdots,n-1\}$,
\begin{eqnarray*}
F_{t,k}(z,v)=\frac{1}{\sqrt{1+t}}\sqrt{\sum_{l=1}^n\frac{|v^l|^2}{(1-|z^l|^2)^2}+t\sqrt[k]{\Big(\sum_{l=1}^{j}\frac{|v^l|^{2}}{(1-|z^l|^2)^2}\Big)^k+\Big(\sum_{l=j+1}^n\frac{|v^l|^2}{(1-|z^l|^2)^2}\Big)^k}},\forall (z,v)\in T^{1,0}P_n
\end{eqnarray*}
 is also a strongly convex K\"ahler-Berwald metric on $P_n$ which is non-Hermitian quadratic. But $F_{t,k}$ is not $\mbox{Aut}(P_n)$-invariant whenever $n\geq 3$.
\end{remark}

\begin{theorem}\label{bh}
Let $M_1$ and $M_2$ be two complex manifolds and $M_1$ is biholomorphically equivalent to $M_2$. Then $M_2$ admits an $\mbox{Aut}(M_2)$-invariant strongly pseudoconvex complex Finsler metric iff $M_1$ admits an $\mbox{Aut}(M_1)$-invariant strongly pseudoconvex complex Finsler metric. More precisely, the following assertions hold:

(1) $M_2$ admits an $\mbox{Aut}(M_2)$-invariant Hermitian metric iff $M_1$ admits an $\mbox{Aut}(M_1)$-invariant Hermitian metric;

(2) $M_2$ admits an non-Hermitian quadratic $\mbox{Aut}(M_2)$-invariant strongly pseudoconvex complex Finsler metric iff $M_1$ admits an non-Hermitian quadratic $\mbox{Aut}(M_1)$-invariant strongly pseudoconvex complex Finsler metric;

(3) $M_2$ admits an $\mbox{Aut}(M_2)$-invariant complex Berwald metric (resp. weakly complex Berwald metric) iff $M_1$ admits an $\mbox{Aut}(M_1)$-invariant complex Berwald metric (resp.  weakly complex Berwald metric);

(4) $M_2$ admits an $\mbox{Aut}(M_2)$-invariant K\"ahler-Finsler metric (resp. weakly K\"ahler-Finsler metric)  iff $M_1$ admits an $\mbox{Aut}(M_1)$-invariant K\"ahler-Finsler metric (resp. weakly K\"ahler-Finsler metric).
\end{theorem}

The following theorem is a refinement of Theorem 5.25 in \cite{Deng}.
\begin{theorem}\label{mth-4}
Let $M_l=G_l/H_l$ be Hermitian symmetric spaces endowed with $G_l$-invariant Hermitian metrics $Q_l$ for $l=1,\cdots,n$ and $M=(G_1/H_1)\times \cdots\times (G_n/H_n)$ the product manifold. For any fixed $t\in[0,+\infty)$ and integer $k\geq 2$, define
\begin{equation}
F_{t,k}(z,v)=\frac{1}{\sqrt{1+t}}\sqrt{\sum_{l=1}^nQ_l(\pmb{\pi}_l(z),(\pmb{\pi}_l)_\ast(v))+t\sqrt[k]{\sum_{l=1}^nQ_l^k(\pmb{\pi}_l(z),(\pmb{\pi}_l)_\ast(v))}},\quad \forall (z,v)\in T^{1,0}M.\label{hsm}
\end{equation}
 Then

(1) $F_{t,k}$ is a strongly convex K\"ahler-Berwald metric $F_{t,k}$ on $M$;

(2) $F_{t,k}$ is invariant under $G_1\times\cdots\times G_n$ and makes $(M,F_{t,k})$ a symmetric complex Finsler space;

(3) if $K_l\equiv c\geq 0$, then $K_{t,k}\in\left[\frac{(1+t)c}{n+t\sqrt[k]{n}},c\right]$;

(4) if $K_l\equiv c<0$, then $K_{t,k}\in \left[c,\frac{(1+t)c}{n+t\sqrt[k]{n}}\right]$.
\end{theorem}

The following theorem gives a characterization of strongly convex K\"ahler-Berwald spaces.

\begin{theorem}\label{mth-c}
A connected, simply connected, complete, strongly convex K\"ahler-Berwald space $(M,F)$ must be one of the following four types:

1) $(M,F)$ is a Hermitian space.

2) $(M,F)$ is a locally complex Minkowskian space.

3) $(M,F)$ is a locally irreducible and locally symmetric complete non-Hermitian K\"ahler-Berwald space of rank $r\geq 2$.

4) $(M,F)$ is locally reducible, and in this case $(M,F)$ can be locally decomposed into a Descartes product of Hermitian  spaces,
locally complex Minkowski spaces and locally irreducible symmetric complete non-Hermitian K\"ahler-Berwald spaces of rank $r\geq 2$.
\end{theorem}

The following theorem gives a de Rahm type decomposition theorem for strongly convex K\"ahler-Berwald spaces, which is a natural generalization of the de Rahm decomposition theorem for K\"ahler manifolds (cf. Theorem 8.1 in vol. II, \cite{KN}).
\begin{theorem}
A connected, simply connected, complete, strongly convex K\"ahler-Berwald space $(M,F)$
can be decomposed into the Descartes product of a complex Minkowski space,
simply connected complete irreducible Hermitian spaces and simply connected complete irreducible globally
 symmetric non-Hermitian strongly convex K\"ahler-Berwald space of rank $r\geq 2$. Such a decomposition is unique up to an order.
\end{theorem}

\section{Complex Finsler metrics}

In this section we first recall some notions of complex Finsler geometry \cite{ap}, and then give a rigidity property of strongly pseudoconvex complex Finsler metrics on the unit ball $B_n$ in $\mathbb{C}^n$.

Let $M$ be a complex manifold of complex dimension $n$. We denote $T^{1,0}M$ the holomorphic tangent bundle of $M$ and $\pi:T^{1,0}M\rightarrow M$  the natural projection. Note that $T^{1,0}M$ is a complex manifold of complex dimension $2n$. A point in $T^{1,0}M$ can be represented by $v\in T_{\pi(v)}^{1,0}M$ so that we can denote $(p,v)$ a point in $T^{1,0}M$ with the understanding that $p\in M$ and $v\in T_p^{1,0}M$. Locally, let $U\subset M$ be a coordinate neighborhood of $p$ such that $z=(z^1,\cdots,z^n)$ are local holomorphic coordinates on $U$. Then a vector $v\in T_p^{1,0}M$ is expressed as $v=v^\alpha\frac{\partial}{\partial z^\alpha}|_{p}$ (where here and in the following we use the Einstein summation convention). Thus we can denote $(z,v)=(z^1,\cdots,z^n,v^1,\cdots,v^n)$ the local holomorphic coordinates (as well as points if it causes no confusion) on the open subset $\pi^{-1}(U)\subset T^{1,0}M$. We call $z$ the base coordinates (or points) on $M$ while $v$ the fiber coordinates (or tangent directions). For a complex manifold $M$ we always denote $\widetilde{M}$ the complement of the zero section in $T^{1,0}M$, that is, $\widetilde{M}=T^{1,0}M\setminus\{\mbox{zero section}\}$.
\begin{definition}[\cite{ap}]\label{cd}
A strongly pseudoconvex complex Finsler metric on $M$ is a continuous function
$$F: T^{1,0}M\rightarrow [0,+\infty)$$
 satisfying

(i) $F(z,v)\geq 0$ for any $(z,v)\in T^{1,0}M$, the equality holds iff $v=0$;

(ii) $G=F^2(z,v)$ is $C^\infty$ for any $(z,v)\in \widetilde{M}$;

(iii) $G(z,\lambda v)=|\lambda|^2G(z,v)$ for any $\lambda\in\mathbb{C}$ and $(z,v)\in T^{1,0}M$;

(iv) the Hermitian matrix $(G_{\alpha\overline{\beta}}):=\Big(\frac{\partial^2G}{\partial v^\alpha\partial\overline{v^\beta}}\Big)$ is  positive definite for any $(z,v)\in\widetilde{M}$.
\end{definition}
\begin{remark}
The condition (iv) is equivalent to the strongly pseudoconvexity of the indicatrix
$$I_F(p)=\{v\in T_p^{1,0}M|F(p,v)<1\}$$
of $F$ at any $p\in M$, which is independent of the choice of local holomorphic coordinates on $M$.
\end{remark}
\begin{remark}
A strongly pseudoconvex complex Finsler metric is $C^k(k\geq 2)$ on $T^{1,0}M$ iff $F$ is a $C^k(k\geq 2)$ Hermitian metric on $M$ (cf. \cite{ap}), in this case we call $F$ is a Hermitian quadratic complex Finsler metric. Thus for a non-Hermitian quadratic complex Finsler metric $F$,  it is $C^\infty$ iff it is $C^\infty$ on $T^{1,0}M\setminus\{\mbox{zero section}\}$.
\end{remark}
\begin{remark}
If locally $F$ is independent of any $z\in M$, that is, $F(z,v)=f(v)$ for some $f$ satisfying the conditions (i)-(iv) in Definition \ref{cd} and any $(z,v)\in \pi^{-1}(U)$, then $(M,F)$ is called a locally complex Minkowski space. If furthermore, $z=(z^1,\cdots,z^n)$ is a global holomorphic coordinates on $M$, then $(M,F)$ is called a complex Minkowski space (cf. \cite{Aikou1}).
\end{remark}

If we denote
$$
z^\alpha=x^\alpha+\sqrt{-1}x^{\alpha+n},\quad v^\alpha=u^\alpha+\sqrt{-1}u^{\alpha+n},\quad \alpha=1,\cdots,n.
$$
Then $(x,u)=(x^1,\cdots,x^{2n}, u^1,\cdots,u^{2n})$ are local real coordinates on the real tangent bundle $TM$ which is $\mathbb{R}$-isomorphism with $T^{1,0}M$.  Using $F$ one may define a function $F^o(x,u)$ for any $(x,u)\in TM$ such that
\begin{equation}
F^o(x,u):=F(z,v).\label{rf}
\end{equation}
In the following we use $G$ to denote the square of the Finsler metric $F$ or $F^o$ (since by \eqref{rf} it is the same) and make the convention that greek indices run from $1$ to $n$, lattin indices run from $1$ to $2n$.

If
\begin{equation}
(g_{ab}):=\Big(\frac{1}{2}\frac{\partial^2G}{\partial u^a\partial u^b}\Big)\label{rc}
\end{equation}
is positive definite for any $(x,u)\in \widetilde{M}$ (here we use the same notion $\widetilde{M}$ to denote the complement of the zero section in $TM$ since $T^{1,0}M$ and $TM$ are $\mathbb{R}$-isomorphism), then $F$ is called strongly convex \cite{ap}.

\begin{remark}
Condition (iii) in Definition \ref{cd} implies that
\begin{equation}
G(z,v)=G_{\alpha\overline{\beta}}(z,v)v^\alpha\overline{v^\beta}.\label{oh}
\end{equation}
A $C^\infty$ Hermitian metric is
 $$ F(z,v)=\sqrt{h_{\alpha\overline{\beta}}(z)v^\alpha \overline{v^\beta}},$$
where $h_{\alpha\overline{\beta}}(z)$ is a $C^\infty$ Hermitian tensor and $(h_{\alpha\overline{\beta}})$ is positive definite on $M$. It is easy to check that a $C^\infty$ Hermitian metric is necessary a strongly convex complex Finsler metric.
 According to S. S. Shern \cite{SSC}, Finsler geometry is just Riemannian geometry without the quadratic restriction. Especially, complex Finsler geometry is  Hermitian geometry without the Hermitian quadratic restriction, which mainly dealing with the differential geometry properties of non-Hermitian quadratic metrics.
In view of \eqref{rc} and \eqref{oh}, the first fundamental form $ds^2$ of a strongly convex complex Finsler metric $F$ can be expressed as
\begin{equation}
ds^2=F^2(z,dz)=G_{\alpha\overline{\beta}}(z,dz)dz^\alpha d\overline{z^\beta}=g_{ab}(x,dx)dx^adx^b.
\end{equation}
\end{remark}

For a strongly pseudoconvex complex Finsler metric $F$ on $M$, the Chern-Finsler non-linear connection coefficients $\varGamma_{;\mu}^\alpha$ associated to $F$ are give by \cite{ap}
\begin{equation}
\varGamma_{;\mu}^\alpha(z,v)=\sum_{\tau=1}^nG^{\overline{\tau}\alpha}\frac{\partial^2G}{\partial \overline{v^\tau}\partial z^\mu},
\end{equation}
where $(G^{\overline{\tau}\alpha})$ is the inverse matrix of $(G_{\alpha\overline{\beta}})$. Note that $\varGamma_{;\mu}^\alpha$
 satisfy
\begin{equation}
\varGamma_{;\mu}^\alpha(z;\lambda v)=\lambda\varGamma_{;\mu}^\alpha(z,v),\quad\forall \lambda\in\mathbb{C}\setminus\{0\}\label{fnc}
\end{equation}
and
\begin{equation}
\varGamma_{;\mu}^\alpha(z,v)=\dot{\partial}_\nu(\varGamma_{;\mu}^\alpha)=\varGamma_{\nu;\mu}^\alpha(z,v)v^\nu,
\end{equation}
 where $\dot{\partial}_\nu=\partial/\partial v^\nu$ and $\varGamma_{\nu;\mu}^\alpha(z,v)$ are just the horizontal Chern-Finsler connection coefficients associated to $F$ (cf. \cite{ap}). Note that \eqref{fnc} does not mean that $\varGamma_{;\mu}^\alpha$ are complex linear with respect to $v$, and in general $\varGamma_{\nu;\mu}^\alpha$ depend on both $z$ and $v$. The horizontal Chern-Finsler connection is also called the complex Rund connection in \cite{Aikou1}, and called Finsler connection in \cite{Aikou2}.
\begin{remark}
If $F=\sqrt{h_{\alpha\overline{\beta}}(z)v^\alpha\overline{v^\beta}}$ is a Hermitian metric, then it is easy to check that
$$
\varGamma_{;\mu}^\alpha=\varGamma_{\nu;\mu}^\alpha(z)v^\nu,\quad \varGamma_{\nu;\mu}^\alpha=h^{\overline{\gamma}\alpha}(z)\frac{\partial h_{\nu\overline{\gamma}}(z)}{\partial z^\mu}.
$$
In this case, $\varGamma_{;\mu}^\alpha$ are complex linear with respect to $v$, $\varGamma_{\nu;\mu}^\alpha=\varGamma_{\nu;\mu}^\alpha(z)$ are independent of $v$, and $\varGamma_{\nu;\mu}^\alpha(z)$ are just the Hermitian connection coefficients associated to $F$.
\end{remark}

\begin{definition} [\cite{ap}]\label{kcd} A strongly pseudoconvex complex Finsler manifold $(M,F)$ is called a strongly K\"ahler-Finsler manifold iff
\begin{equation}
\varGamma_{\nu;\mu}^\alpha-\varGamma_{\mu;\nu}^\alpha=0;\label{sk}
\end{equation}
called a K\"ahler-Finsler manifold iff
\begin{equation}
(\varGamma_{\nu;\mu}^\alpha-\varGamma_{\mu;\nu}^\alpha)v^\nu=0;\label{k}
\end{equation}
called a weakly K\"ahler-Finsler manifold iff
\begin{equation}
G_\alpha(\varGamma_{\nu;\mu}^\alpha-\varGamma_{\mu;\nu}^\alpha)v^\nu=0.\label{wk}
\end{equation}
\end{definition}
\begin{remark}
Although the connection coefficients $\varGamma_{\nu;\mu}^\alpha$ are not globally defined on $\widetilde{M}$, \eqref{sk}-\eqref{wk} are different contractions of the torsion of the Chern-Finsler connection associated to $F$,  thus Definition \ref{kcd} does not depend on the choice of local holomorphic coordinates $(z,v)$ on $T^{1,0}M$.
\end{remark}

In \cite{CS}, B. Chen and Y. Shen proved that $F$ satisfies \eqref{k} implies that $F$ satisfies \eqref{sk}, thus leaving two notions of K\"ahlerian analogue in complex Finsler setting: K\"ahler-Finsler and weakly K\"ahler-Finsler conditions.

Denote
\begin{equation}
\mathbb{G}^\alpha=\frac{1}{2}\varGamma_{;\mu}^\alpha v^\mu,\quad \mathbb{G}_{\;\mu}^\alpha=\dot{\partial}_\mu(\mathbb{G}^\alpha),\quad \mathbb{G}_{\nu\mu}^\alpha=\dot{\partial}_\nu\dot{\partial}_\mu(\mathbb{G}^\alpha).
\end{equation}
Then a strongly pseudoconvex complex Finsler metric $F$ on $M$ is called a complex Berwald metric  if
$\varGamma_{\nu;\mu}^\alpha$
are independent of fiber coordinates $v$ (cf. \cite{Aikou1}); called a weakly complex Berwald metric if $\mathbb{G}^\alpha$ are holomorphic and quadratic with respect to fiber coordinates $v$ (equivalently,
$\mathbb{G}_{\nu\mu}^\alpha$ are independent of fiber coordinates $v$)(cf. \cite{Zh1}). Note that the definition of complex Berwald metric and weakly complex Berwald metric are also independent of the choice of local holomorphic coordinates $(z,v)$ on $T^{1,0}M$. There are indeed lots of weakly complex Berwald metrics which are not complex Berwald metrics \cite{Zh2}. In \cite{Zh3}, Xia and Zhong gave a classification of $U(n)$-invariant weakly complex Berwald metrics with constant holomorphic sectional curvatures.

\subsection{Holomorphic invariant complex Finsler metric on the unit ball}

 It is well-know that the unit disk $\triangle=\{|z|<1\}$ in $\mathbb{C}$ admits an $\mbox{Aut}(\triangle)$-invariant complete K\"ahler metric, that is, the Poincar$\acute{\mbox{e}}$ metric $\frac{|dz|^2}{(1-|z|^2)^2}$ with constant holomorphic curvature $-4$. The following theorem  shows that on the unit ball $B_n$ in $\mathbb{C}^n$ there actually exists no holomorphic invariant strongly pseudoconvex complex Finsler metric other than a constant multiple of the Poincar$\acute{\mbox{e}}$-Bergman metric.
\begin{theorem}\label{nb}
The unit ball $B_n=\{\|z\|^2<1\}$ in $\mathbb{C}^n$ admits no $\mbox{Aut}(B_n)$-invariant strongly pseudoconvex complex Finsler metric other than a constant multiple of the Poincar$\acute{\mbox{e}}$-Bergman metric.
\end{theorem}
\begin{proof}
By Theorem 1.4 in \cite{Zhu},  each $\varphi\in\mbox{Aut}(B_n)$ is of the form $\varphi=U\circ \varphi_a=\varphi_b\circ V$, where $U,V\in U(n)$ are unitary transformations of $\mathbb{C}^n$ and
$$
\varphi_a(z)=\frac{a-\frac{\langle z,a\rangle}{\|a\|^2}a-\sqrt{1-\|a\|^2}\Big(z-\frac{\langle z,a\rangle}{\|a\|^2}a\Big)}{1-\langle z,a\rangle}\in\mbox{Aut}(B_n), \quad 0\neq a\in B_n,
$$
are involutions which interchange the point $a\in B_n$ and $0$, that is, $\varphi_a\circ\varphi_a(z)=z$ and $\varphi_a(0)=a,\varphi_a(a)=0$.

Now suppose that $F(z,v)$ is any $\mbox{Aut}(B_n)$-invariant strongly pseudoconvex complex Finsler metric. Then restricting to the holomorphic tangent space $T_0^{1,0}B_n\cong\mathbb{C}^n$ at the origin $0\in B_n$, $f_0^2(v):=F^2(0,v)$ is a complex Minkowski norm, which should be $U(n)$-invariant, that is $f_0^2(Av)=f_0^2(v)$ for any unitary matrix $A$, this is because the isotropic subgroup of $\mbox{Aut}(B_n)$ at the origin $0\in B_n$ is just the unitary transformations $U(n)$ of $\mathbb{C}^n$. On the other hand, by Theorem 2.1 in \cite{Zh3}, the square of any
$U(n)$-invariant
strongly pseudoconvex complex Minkowski norm is necessary of the form $c\|\cdot\|^2$  where $c=\phi(0,0)>0$ is a positive constant and $\|\cdot\|^2$ is the canonical complex Euclidean norm on $\mathbb{C}^n$. That is, we must have $f_0^2(v)=c\|v\|^2$.

Denote by $\varphi_a(z)=(\varphi_a^1(z),\cdots,\varphi_a^n(z))$, then
\begin{eqnarray*}
\varphi_a^i(z)&=&\frac{a^i-\frac{\langle z,a\rangle}{\|a\|^2}a^i-s_a\Big(z^i-\frac{\langle z,a\rangle}{\|a\|^2}a^i\Big)}{1-\langle z,a\rangle},\\
d\varphi_a^i(z)|_{z=a}&=&-\frac{1}{1-\|a\|^2}\Bigg\{\frac{\langle dz,a\rangle}{\|a\|^2}a^i+\sqrt{1-\|a\|^2}\Big(dz^i-\frac{\langle dz,a\rangle}{\|a\|^2}a^i\Big)\Bigg\}.
\end{eqnarray*}
Since $\mbox{Aut}(B_n)$ acts transitively on $B_n$ and $\varphi_a(a)=0$ for any $a\neq 0$,  we have
$$
F^2(a,dz)=F^2(0,d\varphi_a(z)|_{z=a})=f_0^2(d\varphi_a(z)|_{z=a})=c\cdot \frac{(1-\|a\|^2)|dz|^2+|\langle a,dz\rangle|^2}{(1-\|a\|^2)^2}.
$$
Equivalently for any $z\in B_n$ and $v\in T_z^{1,0}B_n$, we have
\begin{eqnarray*}
F^2(z,v)=c\cdot \frac{(1-\|z\|^2)|v|^2+|\langle z,v\rangle|^2}{(1-\|z\|^2)^2}.
\end{eqnarray*}
This completes the proof.

\end{proof}

\section{Strongly convex complex Finsler metrics on product manifolds}

Note that that a connected complex Berwald manifold $(M,F)$ is also called a complex manifold modeled on a complex Minkowski space \cite{Aikou1}. The importance of complex Berwald manifolds is that one can define parallel displacement of vector fields along curves on $M$ so that parallel displacements are linear isometry of $F$.
A strongly pseudoconvex complex Finsler manifold $(M,F)$ is called a K\"ahler-Berwald manifold if $F$ is both a K\"ahler-Finsler metric and a complex Berwald metric on $M$.

In this section, we shall give an answer to \textbf{Question 1} by proving Theorem \ref{mth-1}.
For this purpose, let's first give some notations of product complex manifolds.

\subsection{Product of complex manifolds}

For each fixed $l\in \{1,\cdots,n\}$, we assume that $M_l$ is a complex manifold of complex dimension $m_l$. Then the product manifold $M=M_1\times \cdots\times M_n$ is of complex dimension $N=m_1+\cdots+m_n$ with
 the natural projections
\begin{equation}
\pmb{\pi}_l: M\rightarrow M_l\label{pl}
\end{equation}
of $M$ onto the $l$-th factor $M_l$
and the natural embeddings
\begin{equation}
\pmb{i}_l:M_l\rightarrow M\label{il}
\end{equation}
 such that for any fixed point $p=(p_1,\cdots,p_n)\in M$, we have
$$\pmb{\pi}_l(p)=p_l\in M_l$$
and
\begin{equation*}
\pmb{i}_l(q)=(\underbrace{p_1,\cdots,p_{l-1}}_{l-1}, q,\underbrace{p_{l+1},\cdots,p_n}_{n-l})\in M, \quad \forall q\in M_l
\end{equation*}
for $l=1,\cdots,n$.

Thus for $p=(p_1,\cdots,p_n)\in M$, the holomorphic tangent space $T_p^{1,0}M$ can be identified with
$$T_{p_1}^{1,0}M_1\oplus \cdots \oplus T_{p_n}^{1,0}M_n.$$
Namely, let $v_l\in T_{p_l}^{1,0}M_l$, choose curves $z_l(t)$ such that $v_l$ is tangent to $z_l(t)$ at $p_l=z_l(t_0)$, then
$$(v_1,\cdots,v_n)\in T_{p_1}^{1,0}M_1\oplus \cdots\oplus T_{p_n}^{1,0}M_n$$
is identified with the vector $v\in T_p^{1,0}M$, which is tangent to the curve $z(t)=(z_1(t),\cdots,z_n(t))$ at the point $p=(p_1,\cdots,p_n)=(z_1(t_0),\cdots,z_n(t_0))$.
Conversely, let $\widetilde{v_l}\in T_p^{1,0}M$ be the vector tangent to the curve $(p_1,\cdots,p_{l-1}, z_l(t),p_{l+1},\cdots,p_n)$ in $M$ at the point $p=(p_1,\cdots,p_{l-1},p_l,p_{l+1},\cdots,p_n)$ for $l=1,\cdots,n$, that is
$$
\widetilde{v_l}=(\pmb{i}_l)_\ast(v_l),\quad \forall l=1,\cdots,n,
$$
then it is easy to check that $v=\widetilde{v_1}+\cdots+\widetilde{v_n}$.

It follows that if $v\in T_p^{1,0}M$ corresponds to $(v_1,\cdots,v_n)\in T_{p_1}^{1,0}M_1\oplus \cdots\oplus T_{p_n}^{1,0}M_n$ at the point $p=(p_1,\cdots,p_n)\in M$, then
\begin{equation}
(\pmb{\pi}_l)_\ast(v)=(\pmb{\pi}_l\circ \pmb{i}_1)_\ast (v_1)+\cdots+(\pmb{\pi}_l\circ \pmb{i}_n)_\ast(v_n)=(\pmb{\pi}_l\circ\pmb{i}_l)_\ast(v_l)=v_l,\quad \forall l=1,\cdots,n.
\end{equation}

In order to prove Theorem \ref{mth-1}, we need local coordinates on $M$ and its factors $M_1,\cdots,M_n$, respectively. Thus for each fixed $l\in \{1,\cdots,n\}$, we denote $z_l=(z_l^1,\cdots,z_l^{m_l})$ the local holomorphic coordinates  on a neighborhood  $U_l\ni p_l\in M_l$ such that
$$z_l^{\alpha_l}=x_l^{\alpha_l}+ix_l^{m_l+\alpha_l},\quad \alpha_l=1,\cdots,m_l.$$
Then $x_l=(x_l^1,\cdots,x_l^{2m_l})$ are the corresponding local real analytic coordinates on $U_l\ni p_l\in M_l$ for $l=1,\cdots,n$. So that
$$z=(z_1,\cdots,z_n)=(z_1^1,\cdots,z_1^{m_1},\cdots,z_n^1,\cdots,z_n^{m_n})$$ are local holomorphic coordinates on $U_1\times \cdots\times U_n\ni p=(p_1,\cdots,p_n)\in M$
 and
$$
x=(x_1,\cdots,x_n)=(x_1^1,\cdots,x_1^{2m_1},\cdots,x_n^1,\cdots,x_n^{2m_n})
$$
are the corresponding local real analytic coordinates. In the following we also use $z_l$ (or $x_l$), $z$ (or $x$)  to denote the corresponding points on $M_l$ and $M$, respectively if they cause no confusions.

For each fixed $l\in\{1,\cdots,n\}$, we denote $\pi_l:T^{1,0}M_l\rightarrow M_l$ the natural projections and $v_l\in T_{z_l}^{1,0}M_l$ (resp. $u_l\in T_{x_l}M$) the holomorphic vector (resp. real tangent vector) at the point $z_l\in M_l$ (resp. $x_l\in M_l$). Setting
\begin{eqnarray*}
v_l^{\alpha_l}&=&u_l^{\alpha_l}+iu_l^{\alpha_l+m_l},\quad \alpha_l=1,\cdots,m_l
\end{eqnarray*}
for each fixed $l\in\{1,\cdots,n\}$.
 Then $(z_l,v_l)=(z_l^1,\cdots,z_l^{m_l},v_l^1,\cdots,v_l^{m_l})$ are local holomorphic coordinates on $\pi_l^{-1}(U_l)\subset T^{1,0}M_l$ for $l=1,\cdots,n$, so that
 $$
 (z,v)=(z_1^1,\cdots,z_1^{m_1},\cdots,z_n^1,\cdots,z_n^{m_n},v_1^1,\cdots,v_1^{m_1},\cdots,v_n^1,\cdots,v_n^{m_n})
 $$
 are local holomorphic coordinates on $\pi_1^{-1}(U_1)\times \cdots \times\pi_n^{-1}(U_n)\subset T^{1,0}M$ and
$$
(x,u)=(x_1^1,\cdots,x_1^{2m_1},\cdots,x_n^1,\cdots,x_n^{2m_n},u_1^1,\cdots,u_1^{2m_1},\cdots,u_n^1,\cdots,u_n^{2m_n})
$$
are the corresponding local real analytic coordinates.
 In terms of local real and complex coordinates on  $\pi_l^{-1}(U_l)\subset T^{1,0}M_l$, the Hermitian metric $Q_l$ on $M_l$ can be represented as
\begin{equation}
Q_l(z_l,v_l)=\sum_{\alpha_l,\beta_l=1}^{m_l}[Q_l]_{\alpha_l\overline{\beta_l}}(z_l)v_l^{\alpha_l}\overline{v_l^{\beta_l}},\quad v_l\in T_{z_l}^{1,0}M_l
\end{equation}
and
\begin{equation}
Q_l(x_l,u_l)=\frac{1}{2}\sum_{a_l,b_l=1}^{2m_l}[Q_l]_{a_l b_l}(x_l)u_l^{a_l}u_l^{b_l},\quad u_l\in T_{x_l}M_l,\label{rhm}
\end{equation}
respectively, where
\begin{eqnarray*}
{[}Q_l]_{\alpha_l\overline{\beta_l}}(z_l)&=&\frac{\partial^2Q_l}{\partial v_l^{\alpha_l}\partial \overline{v_l^{\beta_l}}},\quad \alpha_l,\beta_l=1,\cdots,m_l,\\
{[}Q_l]_{a_l b_l}(x_l)&=&\frac{\partial^2Q_l}{\partial u_l^{a_l}\partial u_l^{b_l}},\quad a_l,b_l=1,\cdots,2m_l.
\end{eqnarray*}

In the following for each fixed $l\in \{1,\cdots,n\}$, we denote $([Q_l]^{\overline{\beta_l}\gamma_l})$ the inverse matrix of the $m_l\times m_l$ matrix $([Q_l]_{\alpha_l\overline{\beta_l}})$, and $([Q_l]^{b_lc_l})$ the inverse matrix of the $(2m_l)\times (2m_l)$ matrix $([Q_l]_{a_l b_l})$; we denote
\begin{equation}
\breve{\varGamma}_{a_s;b_s}^{c_s}(x_s)=\frac{1}{2}[Q_s]^{c_sd_s}\Big(\frac{\partial [Q_s]_{b_sd_s}}{\partial x_s^{a_s}}+\frac{\partial [Q_s]_{a_sd_s}}{\partial x_s^{b_s}}-\frac{\partial [Q_s]_{a_sb_s}}{\partial x_s^{d_s}}\Big)\label{lcc}
\end{equation}
the Levi-Civita connection coefficients associated to $Q_s$
which satisfy
\begin{equation}
\breve{\varGamma}_{;b_s}^{c_s}(x_s,u_s)=\breve{\varGamma}_{a_s;b_s}^{c_s}(x_s)u_s^{a_s}
\end{equation}
 and
\begin{equation}
\hat{\varGamma}_{\alpha_s;\beta_s}^{\gamma_s}(z_s)=[Q_s]^{\overline{\lambda_s}\gamma_s}\frac{\partial[Q_s]_{\alpha_s\overline{\lambda_s}}}{\partial z_s^{\beta_s}}\label{hcc}
\end{equation}
the Hermitian connection coefficients associated to $Q_l$
which satisfy
\begin{equation}
\hat{\varGamma}_{;\beta_s}^{\gamma_s}(z_s,v_s)=\hat{\varGamma}_{\alpha_s;\beta_s}^{\gamma_s}(z_s)v_s^{\alpha_s}.
\end{equation}
Finally, for each fixed $l\in\{1,\cdots,n\}$ we denote
\begin{equation}
K_l=-\frac{2}{Q_l^2}\frac{\partial Q_l}{\partial v_l^{\gamma_l}}\frac{\partial}{\partial \overline{z_l^{\mu_l}}}(\hat{\varGamma}_{;\alpha_l}^{\gamma_l})v_l^{\alpha_l}\overline{v_l^{\mu_l}}
\end{equation}
the holomorphic sectional curvature of $Q_l$ along $0\neq v_l\in T_{z_l}^{1,0}M_l$.

\subsection{Strongly convex K\"ahler-Berwald metrics on product manifolds}

In this section we are in a position to prove Theorem \ref{mth-1}. We recall it here for convenience.

 \begin{theorem}\label{mtha}
Let $(M,Q)$ be a simply connected complete reducible $C^\infty$ Hermitian manifold (resp. K\"ahler manifold) such that $(M_1,Q_1)\times \cdots\times (M_n,Q_n)$ is the de Rham decomposition of $(M,Q)$ and $\pmb{\pi}_l:M\rightarrow M_l$ are the natural projections of $M$ onto the $l$-th factors $M_l$ for $l=1,\cdots,n$. Then

(1) $M$ admits infinite many strongly convex complete complex Berwald metrics (resp. K\"ahler-Berwald metrics)
\begin{equation}
F_{t,k}(z,v)=\frac{1}{\sqrt{1+t}}\sqrt{\sum_{l=1}^nQ_l(\pmb{\pi}_l(z),(\pmb{\pi}_l)_\ast(v))+t\sqrt[k]{\sum_{l=1}^nQ_l^k(\pmb{\pi}_l(z),(\pmb{\pi}_l)_\ast(v))}},\quad\forall (z,v)\in T^{1,0}M,\label{ftk}
\end{equation}
where $t\in[0,+\infty)$ and $k\geq 2$ is an integer;

(2) $F_{t,k}$ is a complete real Berwald metric on $M$, thus $F_{t,k}$ has vanishing $\textbf{S}$-curvature for any $t\in[0,+\infty)$ and integer $k\geq 2$.
\end{theorem}

Before proving the above theorem, we give some remarks.

\begin{remark}
Denote $\widetilde{M_l}$ the complement of the zero section in $T^{1,0}M_l$, that is, $\widetilde{M_l}=T^{1,0}M_l\setminus\{\mbox{zero section}\}$ for $l=1,\cdots,n$ and $\widetilde{M}=T^{1,0}M\setminus\{\mbox{zero section}\}$. Then it is clear that $F_{0}^2=Q_1+\cdots+Q_n$ is $C^\infty$ on the whole $T^{1,0}M$ while $F_{t,k}^2$ is $C^\infty$ on $\widetilde{M}$ for $t\in(0,+\infty)$ and integer $k\geq 2$.
\end{remark}
\begin{remark}
The factor $\frac{1}{\sqrt{1+t}}$ before $F_{t,k}$ is important for our discussion, we use it so that the natural projections $\pmb{\pi}_l: (M,F_{t,k})\rightarrow (M_l,Q_l)$ and embeddings $\pmb{i}_l:(M_l,Q_l)\rightarrow (M,F_{t,k})$ are holomorphic isometries for $l=1,\cdots,n$.
The assertion that $F_{t,k}$ (without the factor $\frac{1}{\sqrt{1+t}}$) is a strongly convex complex Berwald metric was first proved in \cite{Zh4} for $n=2$. The proof of the general case $n>2$, however, cannot be  simply obtained by induction on $n$, see also the Remark \ref{remark-3}. Note that $F_{t,k}$ is a strongly convex complex Berwald metric clearly implies that $F_{t,k}$ is a strongly pseudoconvex complex Berwald metric \cite{DYH}.
\end{remark}
\begin{remark}
In general, the differential geometry of non-quadratic real and complex Finsler metrics is not tightly related as that of Riemannian and Hermitian metrics. For example, the real Cartan connection and the Chern-Finsler connection associated to a strong convex K\"ahler-Finsler metric are different (cf. \cite{ap}), while in Riemannian geometry it is well-known that the Levi-Civita connection and the Hermitian connection of a K\"ahler metric coincides.  In \cite{Zh1}, the author proved that under the assumption that $(M,F)$ is a strongly convex weakly K\"ahler-Finsler manifold,  $F$ is complex Berwald metric iff $F$ is a real Berwald metric,  if furthermore $(M,F)$ is a strongly convex K\"ahler-Berwald manifold, then the real and complex Berwald connections associated to $F$ coincide (cf. Theorem 1.2 and Theorem 1.4 in \cite{Zh1}).  Theorem \ref{mtha} provides infinite many examples of strongly convex complex Berwald manifolds which are simultaneously real Berwald manifolds without the K\"ahler-Finsler assumption; moreover, it also provides us an effective way to construct the most important class of non-Hermitian strongly convex K\"ahler-Berwald manifolds on which the real Cartan connection and the Chern-Finsler connection associated to $F_{t,k}$ coincide (cf. \cite{Zh1}).
\end{remark}

\begin{proof} By the decomposition theorem of de Rham (cf. Theorem 6.2 of Chapter IV in volume I \cite{KN}), a simply connected, complete, reducible Hermitian manifold $(M,Q)$ is isometric to the direct product
$(M_1,Q_1)\times \cdots \times (Q_n,M_n)$, where $(M_1,Q_1),\cdots, (M_n,Q_n)$ are all simply connected, complete, irreducible Hermitian manifolds and moreover, such a decomposition is unique up to an order. If furthermore $(M,Q)$ is a K\"ahler manifold, then $(M_1,Q_1),\cdots, (M_n,Q_n)$ are all simply connected, complete, irreducible K\"ahler manifolds and the isometry between $M$ and $M_1\times\cdots\times M_n$ is holomorphic (cf. Theorem 8.1 of Chapter IX in volume II \cite{KN}). Here $Q_l$ is the corresponding Hermitian metric (resp. K\"ahler metric) on $M_l$ for $l=1,\cdots,n$.

Now let $F_{t,k}$ be defined by \ref{ftk}. It is clear that $F_{t,k}$ satisfies Definition \ref{cd} (i)-(iii). In the following we
 denote $G=F_{t,k}^2$ and set $\mathcal{A}=Q_1^k+\cdots+Q_n^k$. The proof are divided into $5$ steps.

\textbf{Step 1.} The strongly convexity of $F_{t,k}$.

Note that for any $(z,v)\in \widetilde{M}$, we have $Q_l(\pmb{\pi}_l(z),(\pmb{\pi}_l)_\ast (v))=Q_l(z_l,v_l)$ for $l=1,\cdots,n$. Denote $\mathcal{V}^{1,0}$ the complex vertical subbundle of $T^{1,0}\widetilde{M}$ and let
$$V=\sum_{l=1}^n\sum_{\alpha_l=1}^{m_l}V_l^{\alpha_l}\frac{\partial}{\partial v_l^{\alpha_l}}\in \mathcal{V}_{(z,v)}^{1,0},$$
with $\quad V_l^{\alpha_l}=U_l^{\alpha_l}+\sqrt{-1}U_l^{\alpha_l+m_l}$ for $l=1,\cdots,n$ and $\alpha_l=1,\cdots,m_l$.
By Proposition 2.6.1 in \cite{ap}, we have the following identity:
\begin{equation}
\sum_{l,s=1}^n\sum_{a_l=1}^{2m_l}\sum_{b_s=1}^{2m_s}\frac{\partial^2G}{\partial u_l^{a_l}\partial u_s^{b_s}}U_l^{a_l}U_s^{b_s}=2\mbox{Re}\sum_{l,s=1}^n\sum_{\alpha_l=1}^{m_l}\sum_{\beta_s=1}^{m_s}\Bigg\{\frac{\partial^2G}{\partial v_l^{\alpha_l}\partial\overline{v_s^{\beta_s}}}V_l^{\alpha_l}\overline{V_s^{\beta_s}}+\frac{\partial^2G}{\partial v_l^{\alpha_l}\partial v_s^{\beta_s}}V_l^{\alpha_l}V_s^{\beta_s}\Bigg\}.\label{sc}
\end{equation}
Thus in order to prove the strongly convexity of $F_{t,k}$, it is suffice to prove that the right hand side of \eqref{sc} is non-negative for any $V\in\mathcal{V}_{(z,v)}^{1,0}$ with $(z,v)\in\widetilde{M}$ and equality holds iff $V=0$.

Note that we have
\begin{eqnarray*}
\frac{\partial^2G}{\partial v_l^{\alpha_l}\partial\overline{v_s^{\beta_s}}}
&=&\Bigg\{\mathcal{E}_l[Q_l]_{\alpha_l\overline{\beta_l}}+t(k-1)\mathcal{A}^{\frac{1}{k}-1}Q_l^{k-2}\frac{\partial Q_l}{\partial v_l^{\alpha_l}}\frac{\partial Q_l}{\partial \overline{v_l^{\beta_l}}}\Bigg\}\delta_{ls}\\
&&-t(k-1)\mathcal{A}^{\frac{1}{k}-2}Q_l^{k-1}Q_s^{k-1}\frac{\partial Q_l}{\partial v_l^{\alpha_l}}\frac{\partial Q_s}{\partial \overline{v_s^{\beta_s}}},\\
\frac{\partial^2G}{\partial v_l^{\alpha_l}\partial v_s^{\beta_s}}
&=&t(k-1)\mathcal{A}^{\frac{1}{k}-1}Q_l^{k-2}\frac{\partial Q_l}{\partial v_l^{\alpha_l}}\frac{\partial Q_s}{\partial v_s^{\beta_s}}\delta_{ls}
-t(k-1)\mathcal{A}^{\frac{1}{k}-2}Q_l^{k-1}Q_s^{k-1}\frac{\partial Q_l}{\partial v_l^{\alpha_l}}\frac{\partial Q_s}{\partial v_s^{\beta_s}},
\end{eqnarray*}
where $\delta_{ls}$ is the Kronecker symbol and we denote
\begin{equation}
 \mathcal{E}_l=1+t \mathcal{A}^{\frac{1}{k}-1}Q_l^{k-1}.\label{bl}
\end{equation}

For each $l\in\{1,\cdots,n\}$ setting
$$\langle V_l,v_l\rangle_l=[Q_l]_{\alpha_l\overline{\beta_l}}(z_l)V_l^{\alpha_l}\overline{v_l^{\beta_l}},
$$
then $Q_l=[Q_l]_{\alpha\overline{\beta}}(z_l)v_l^\alpha \overline{v_l^\beta}$ and
$$
\overline{\langle V_l,v_l\rangle_l}=\langle v_l,V_l\rangle_l,\quad \forall l=1,\cdots,n.
$$
Thus we have
\begin{eqnarray*}
&&\sum_{l,s=1}^n\sum_{\beta_l=1}^{m_l}\sum_{\beta_s=1}^{m_s}\Bigg\{\frac{\partial^2G}{\partial v_l^{\alpha_l}\partial\overline{v_s^{\beta_s}}}V_l^{\alpha_l}\overline{V_s^{\beta_s}}+\frac{\partial^2G}{\partial v_l^{\alpha_l}\partial v_s^{\beta_s}}V_l^{\alpha_l}V_s^{\beta_s}\Bigg\}\\
&=&\sum_{l=1}^n\mathcal{E}_l\langle V_l,V_l\rangle_l+t(k-1)\mathcal{A}^{\frac{1}{k}-2}\Bigg\{\mathcal{A}\Big(\sum_{l=1}^nQ_l^{k-2}\langle V_l,v_l\rangle_l\langle v_l,V_l\rangle_l+\sum_{l=1}^nQ_l^{k-2}\langle V_l,v_l\rangle_l^2\Big)\\
&&-\Big(\sum_{l=1}^nQ_l^{k-1}\langle V_l,v_l\rangle_l\sum_{s=1}^nQ_s^{k-1}\langle v_s,V_s\rangle_s+
\sum_{l=1}^nQ_l^{k-1}\langle V_l,v_l\rangle_l\sum_{s=1}^nQ_s^{k-1}\langle V_s,v_s\rangle_s\Big)\Bigg\}.
\end{eqnarray*}
Note that since $(z,v)\in\widetilde{M}$, we have $\mathcal{A}>0,\mathcal{E}_l>0,\langle V_l,V_l\rangle\geq 0$ and $\langle V_l,v_l\rangle_l\langle v_l,V_l\rangle_l\geq 0$ for $l=1,\cdots,n$. It follows that
\begin{eqnarray*}
2\mbox{Re}\Bigg\{\sum_{l=1}^nQ_l^{k-2}\langle V_l,v_l\rangle_l\langle v_l,V_l\rangle_l+\sum_{l=1}^nQ_l^{k-2}\langle V_l,v_l\rangle_l^2\Bigg\}
=\sum_{l=1}^nQ_l^{k-2}\Big[\langle V_l,v_l\rangle_l+\langle v_l,V_l\rangle_l\Big]^2
\end{eqnarray*}
and
\begin{eqnarray*}
&&2\mbox{Re}\Bigg\{\sum_{l=1}^nQ_l^{k-1}\langle V_l,v_l\rangle_l\sum_{s=1}^nQ_s^{k-1}\langle v_s,V_s\rangle_s+
\sum_{l=1}^nQ_l^{k-1}\langle V_l,v_l\rangle_l\sum_{s=1}^nQ_s^{k-1}\langle V_s,v_s\rangle_s\Bigg\}\\
&=&\Bigg(\sum_{l=1}^nQ_l^{k-1}[\langle V_l,v_l\rangle_l+\langle v_l,V_l\rangle_l]\Bigg)^2,
\end{eqnarray*}
from which we get
\begin{eqnarray*}
&&2\mbox{Re}\sum_{l,s=1}^n\sum_{\beta_l=1}^{m_l}\sum_{\beta_s=1}^{m_s}\Bigg\{
\frac{\partial^2G}{\partial v_l^{\alpha_l}\partial\overline{v_s^{\beta_s}}}V_l^{\alpha_l}\overline{V_s^{\beta_s}}+\frac{\partial^2G}{\partial v_l^{\alpha_l}\partial v_s^{\beta_s}}V_l^{\alpha_l}V_s^{\beta_s}
\Bigg\}\\
&=&2\sum_{l=1}^n\mathcal{E}_l\langle V_l,V_l\rangle_l\\
&&+t(k-1)\mathcal{A}^{\frac{1}{k}-2}\Bigg\{
\sum_{1\leq s<l\leq n}\Big[Q_s^kQ_l^{k-2}\Big(\langle V_l,v_l\rangle_l+\langle v_l,V_l\rangle_l\Big)^2+Q_s^{k-2}Q_l^{k}\Big(\langle V_s,v_s\rangle_s+\langle v_s,V_s\rangle_s\Big)^2\Big]\\
&&-2\sum_{1\leq s<l\leq n}Q_l^{k-1}Q_s^{k-1}\Big[\langle V_l,v_l\rangle_l+\langle v_l,V_l\rangle_l\Big]\Big[\langle V_s,v_s\rangle_s+\langle v_s,V_s\rangle_s\Big]\Bigg\}\\
&=&2\sum_{l=1}^n\mathcal{E}_l\langle V_l,V_l\rangle_l\\
&&+t(k-1)\mathcal{A}^{\frac{1}{k}-2}\Bigg\{
\sum_{1\leq s<l\leq n} Q_s^{k-2}Q_l^{k-2}\Big[Q_s(\langle V_l,v_l\rangle_l+\langle v_l,V_l\rangle_l)-Q_l(\langle V_s,v_s\rangle_s+\langle v_s,V_s\rangle_s)\Big]^2
\Bigg\}\\
&\geq& 0,
\end{eqnarray*}
and equality holds iff $V=0$.
Thus $F_{t,k}$ is a strongly convex complex Finsler metric on $M$.

\textbf{Step 2.}  $F_{t,k}$ is a real Berwald metric on $M$.

We prove this by deriving the real geodesic coefficients $\pmb{G}^{b_s}$ associated to $F_{t,k}$ and show that $\pmb{G}^{b_s}$ are quadratic with respect to $u_l=(u_l^1,\cdots,u_l^{2m_l})$ for $l=1,\cdots,n$.
Let's denote $(G_{a_lb_s}):=\Big(\frac{\partial^2G}{\partial u_l^{a_l}\partial u_s^{b_s}}\Big)$ the real fundamental tensor associated to $F_{t,k}$. Then
\begin{eqnarray*}
G_{a_lb_s}&=&\Bigg\{\mathcal{E}_l\frac{\partial^2Q_l}{\partial u_l^{a_l}\partial u_l^{b_l}}+t(k-1)\mathcal{A}^{\frac{1}{k}-1}Q_l^{k-2}\frac{\partial Q_l}{\partial u_l^{a_l}}\frac{\partial Q_l}{\partial u_l^{b_l}}\Bigg\}\delta_{ls}
-t(k-1)\mathcal{A}^{\frac{1}{k}-2}Q_l^{k-1}Q_s^{k-1}\frac{\partial Q_l}{\partial u_l^{a_l}}\frac{\partial Q_s}{\partial u_s^{b_s}},
\end{eqnarray*}
that is, the real fundamental tensor matrix of $F_{t,k}$ is given by the following $2N\times 2N$ matrix:
\begin{eqnarray}
(G_{a_lb_s})
=\begin{pmatrix}
   \frac{\partial^2G}{\partial u_1^1\partial u_1^1} & \cdots & \frac{\partial^2G}{\partial u_1^1\partial u_1^{2m_1}} & \cdots & \frac{\partial^2G}{\partial u_1^1\partial u_n^1} & \cdots & \frac{\partial^2G}{\partial u_1^1\partial u_n^{2m_n}} \\
    \vdots & \vdots & \vdots & \vdots & \vdots & \vdots &\vdots\\
   \frac{\partial^2G}{\partial u_1^{2m_1}\partial u_1^1} & \cdots & \frac{\partial^2G}{\partial u_1^{2m_1}\partial u_1^{2m_1}} & \cdots & \frac{\partial^2G}{\partial u_1^{2m_1}\partial u_n^1} & \cdots & \frac{\partial^2G}{\partial u_1^{2m_1}\partial u_n^{2m_n}} \\
   \vdots & \vdots & \vdots & \vdots & \vdots & \vdots & \vdots \\
   \frac{\partial^2G}{\partial u_n^1\partial u_1^1} & \cdots & \frac{\partial^2G}{\partial u_n^1\partial u_1^{2m_1}} & \cdots & \frac{\partial^2G}{\partial u_n^1\partial u_n^1} & \cdots & \frac{\partial^2G}{\partial u_n^1\partial u_n^{2m_n}} \\
    \vdots & \vdots & \vdots & \vdots & \vdots & \vdots & \vdots \\
   \frac{\partial^2G}{\partial u_n^{2m_n}\partial u_1^1} & \cdots & \frac{\partial^2G}{\partial u_n^{2m_n}\partial u_1^{2m_1}} & \cdots & \frac{\partial^2G}{\partial u_n^{2m_n}\partial u_n^1} & \cdots & \frac{\partial^2G}{\partial u_n^{2m_n}\partial u_n^{2m_n}} \\
 \end{pmatrix}.\label{rf}
\end{eqnarray}

In the following we denote $(G^{b_sc_r})$ the inverse matrix of $(G_{a_lb_s})$.  For this purpose, we shall rewrite it in terms of blocked matrices. For $l=1,\cdots,n$, we denote $\pmb{B}_l=([B_l]_{a_lb_l})$ the $2m_l\times 2m_l$ matrix with
\begin{eqnarray*}
 [B_l]_{a_lb_l}&=&\mathcal{E}_l\Bigg\{[Q_l]_{a_lb_l}+\frac{t(k-1)\mathcal{A}^{\frac{1}{k}-1}Q_l^{k-2}}{1+t \mathcal{A}^{\frac{1}{k}-1}Q_l^{k-1}}\frac{\partial{Q}_l}{\partial u_l^{a_l}}\frac{\partial Q_l}{\partial u_l^{b_l}}\Bigg\}
\end{eqnarray*}
and let
\begin{eqnarray*}
\pmb{Z}&=&\begin{pmatrix}
    \pmb{Z}_1 \\
    \vdots \\
    \pmb{Z}_n \\
  \end{pmatrix},\quad \pmb{Z}^T=(\pmb{Z}_1^T,\cdots,\pmb{Z}_n^T),
\end{eqnarray*}
where
$$
 \pmb{Z}_l=\begin{pmatrix}
                            Q_l^{k-1}\frac{\partial Q_l}{\partial u_l^1} \\
                            \vdots \\
                            Q_l^{k-1}\frac{\partial Q_l}{\partial u_l^{m_l}} \\
                          \end{pmatrix},\quad l=1,\cdots,n.
$$

Then \eqref{rf} can be rewritten as
$$
(G_{a_lb_s})
=\begin{pmatrix}
   \pmb{B}_{1} & \cdots & 0 \\
   \vdots & \ddots & \vdots \\
   0 & \cdots & \pmb{B}_{n} \\
 \end{pmatrix}
 -t(k-1)\mathcal{A}^{\frac{1}{k}-2}\pmb{Z}\pmb{Z}^T.
$$
Now for each fixed $l\in\{1,\cdots,n\}$, we denote $\pmb{B}_l^{-1}=([B_l]^{b_lc_l})$ the inverse matrix of $\pmb{B}_l$. Then using Lemma 6.1 in \cite{Zh1}, we have
\begin{eqnarray}
[B_l]^{b_lc_l}=\frac{1}{\mathcal{E}_l}\Bigg\{
[Q_l]^{b_lc_l}-\frac{t(k-1)\mathcal{A}^{\frac{1}{k}-1}Q_l^{k-2}}{1+t(2k-1)\mathcal{A}^{\frac{1}{k}-1}Q_l^{k-1}}u_l^{b_l}u_l^{c_l}
\Bigg\},\label{ibl}
\end{eqnarray}
where $([Q_l]^{b_lc_l})$ is the inverse matrix of $([Q_l]_{a_lb_l})$. Using \eqref{ibl}, it is easy to check that
 \begin{eqnarray}
\mathcal{C}:&=&1-t(k-1)\mathcal{A}^{\frac{1}{k}-2}\pmb{Z}^T\begin{pmatrix}
   \pmb{B}_{1} & \cdot & 0 \\
   \vdots & \ddots & \vdots \\
   0 & \cdots & \pmb{B}_{n} \\
 \end{pmatrix}^{-1}\pmb{Z}\nonumber\\
 &=&1-\sum_{l=1}^n\frac{2t(k-1)\mathcal{A}^{\frac{1}{k}-2}Q_l^{2k-1}}{1+t(2k-1)\mathcal{A}^{\frac{1}{k}-1}Q_l^{k-1}}\nonumber\\
 &=&\frac{1}{\mathcal{A}}\Bigg\{\sum_{l=1}^nQ_l^k-\sum_{l=1}^n\frac{2t(k-1)\mathcal{A}^{\frac{1}{k}-1}Q_l^{2k-1}}{1+t(2k-1)\mathcal{A}^{\frac{1}{k}-1}Q_l^{k-1}}\Bigg\}\nonumber\\
 &=&\frac{1}{\mathcal{A}}\sum_{l=1}^n\frac{\mathcal{E}_lQ_l^k}{1+t(2k-1)\mathcal{A}^{\frac{1}{k}-1}Q_l^{k-1}}>0\label{cel}
\end{eqnarray}
 for any $(x,u)\in\widetilde{M}$. Now using Lemma 6.1 in \cite{Zh1}, it follows that the inverse matrix $(G^{a_lc_r})$ of $(G_{a_lb_s})$ is given by
\begin{eqnarray*}
(G^{a_lc_r})
 &=&\begin{pmatrix}
   \pmb{B}_{1}^{-1} & \cdots & 0 \\
   \vdots & \ddots & \vdots \\
   0 & \cdots & \pmb{B}_{n}^{-1} \\
 \end{pmatrix}+\frac{t(k-1)\mathcal{A}^{\frac{1}{k}-2}}{\mathcal{C}}
 \begin{pmatrix}
   \pmb{B}_{1}^{-1} & \cdots & 0 \\
   \vdots & \ddots & \vdots \\
   0 & \cdots & \pmb{B}_{n}^{-1} \\
 \end{pmatrix}
 \pmb{Z}\pmb{Z}^T
 \begin{pmatrix}
   \pmb{B}_{1}^{-1} & \cdots & 0 \\
   \vdots & \ddots & \vdots \\
   0 & \cdots & \pmb{B}_{n}^{-1} \\
 \end{pmatrix},\\
\end{eqnarray*}
or equivalently, with $G^{b_sc_r}$ being given by
\begin{eqnarray}
G^{b_sc_r}=[B_s]^{b_sc_s}\delta_{sr}+\frac{t(k-1)\mathcal{A}^{\frac{1}{k}-2}}{\mathcal{C}}
W_s^{b_s}W_r^{c_r},\label{igsr}
\end{eqnarray}
with
\begin{eqnarray}
W_s^{b_s}
&=&\frac{1}{\mathcal{E}_s}\Bigg\{[Q_s]^{b_sa_s}-\frac{t(k-1)\mathcal{A}^{\frac{1}{k}-1}Q_s^{k-2}}{1+t(2k-1)\mathcal{A}^{\frac{1}{k}-1}Q_s^{k-1}}\Bigg\}Q_s^{k-1}\frac{\partial Q_s}{\partial u_s^{a_s}}\nonumber\\
&=&\frac{1}{\mathcal{E}_s}\Bigg\{u_s^{b_s}-\frac{2t(k-1)\mathcal{A}^{\frac{1}{k}-1}Q_s^{k-1}u_s^{b_s}}{1+t(2k-1)\mathcal{A}^{\frac{1}{k}-1}Q_s^{k-1}}\Bigg\}Q_s^{k-1}\nonumber\\
&=&\frac{Q_s^{k-1}u_s^{b_s}}{1+t(2k-1) \mathcal{A}^{\frac{1}{k}-1}Q_s^{k-1}}.\label{ws}
\end{eqnarray}
Notice that
\begin{eqnarray*}
\frac{\partial G}{\partial u_r^{c_r}}&=&\mathcal{E}_r\frac{\partial Q_r}{\partial u_r^{c_r}},\quad
\frac{\partial G}{\partial x_r^{c_r}}=\mathcal{E}_r\frac{\partial Q_r}{\partial x_r^{c_r}}=\sum_{l=1}^n\mathcal{E}_r\frac{\partial Q_r}{\partial x_l^{c_l}}\delta_{rl},
\end{eqnarray*}
thus
\begin{eqnarray}
\sum_{l=1}^n\frac{\partial^2G}{\partial u_r^{c_r}\partial x_l^{a_l}}u_l^{a_l}
&=&\Bigg[\mathcal{E}_r\frac{\partial^2Q_r}{\partial u_r^{c_r}\partial x_l^{a_l}}
+t(k-1)\mathcal{A}^{\frac{1}{k}-1}Q_r^{k-2}\frac{\partial Q_r}{\partial u_r^{c_r}}\frac{\partial Q_r}{\partial x_l^{a_l}}\Bigg]u_l^{a_l}\delta_{rl}\nonumber\\
&&-t(k-1)\mathcal{A}^{\frac{1}{k}-2}Q_r^{k-1}Q_l^{k-1}\frac{\partial Q_r}{\partial u_r^{c_r}}\frac{\partial Q_l}{\partial x_l^{a_l}}u_l^{a_l}.\label{prl}
\end{eqnarray}
It follows that the real geodesic spray coefficients $\pmb{G}^{b_s}$ of $F_{t,k}$ are given by (here we use the formula in page 28 in \cite{ap}, which is different from formula defined in \cite{BCS} and \cite{Shen} by a factor $\frac{1}{2}$ because of the choice of the fundamental tensor are different by a factor $\frac{1}{2}$)
\begin{equation}
2\pmb{G}^{b_s}=\sum_{r,l=1}^n\sum_{a_l=1}^{2m_l}\sum_{c_r=1}^{2m_r}G^{b_sc_r}\Big(\frac{\partial^2G}{\partial u_r^{c_r}\partial x_l^{a_l}}u_l^{a_l}-\frac{\partial G}{\partial x_r^{c_r}}\Big).\label{gec}
\end{equation}
Substituting \eqref{igsr} and \eqref{prl} into \eqref{gec}, we have
\begin{eqnarray*}
2\pmb{G}^{b_s}
&=&\mathcal{E}_s\sum_{c_s=1}^{2m_s}[B_s]^{b_sc_s}\Bigg(\sum_{a_s=1}^{2m_s}\frac{\partial^2Q_s}{\partial u_s^{c_s}\partial x_s^{a_s}}u_s^{a_s}-\frac{\partial Q_s}{\partial x_s^{c_s}}\Bigg)\\
&&+t(k-1)\mathcal{A}^{\frac{1}{k}-1}Q_s^{k-2}\Big(\sum_{c_s=1}^{2m_s}[B_s]^{b_sc_s}\frac{\partial Q_s}{\partial u_s^{c_s}}\Big)\Big(\sum_{a_s=1}^{2m_s}\frac{\partial Q_s}{\partial x_s^{a_s}}u_s^{a_s}\Big)\\
&&-t(k-1)\mathcal{A}^{\frac{1}{k}-2}Q_s^{k-1}\Big(\sum_{c_s=1}^{2m_s}[B_s]^{b_sc_s}\frac{\partial Q_s}{\partial u_s^{c_s}}\Big)\Big(\sum_{l=1}^n\sum_{a_l=1}^{2m_l}Q_l^{k-1}\frac{\partial Q_l}{\partial x_l^{a_l}}u_l^{a_l}\Big)
\\
&&+\frac{t(k-1)\mathcal{A}^{\frac{1}{k}-2}}{\mathcal{C}}W_s^{b_s}
\sum_{r=1}^n\mathcal{E}_r\sum_{c_r=1}^{2m_r}W_r^{c_r}
\Big(\sum_{a_r=1}^{2m_r}\frac{\partial^2 Q_r}{\partial u_r^{c_r}\partial x_r^{a_r}}u_r^{a_r}-\frac{\partial Q_r}{\partial x_r^{c_r}}\Big)\\
&&+\frac{t(k-1)\mathcal{A}^{\frac{1}{k}-2}}{\mathcal{C}}W_s^{b_s}
\sum_{r=1}^nt(k-1)\mathcal{A}^{\frac{1}{k}-1}Q_r^{k-2}\sum_{c_r=1}^{2m_r}W_r^{c_r}
\Big(\sum_{a_r=1}^{2m_r}\frac{\partial Q_r}{\partial u_r^{c_r}}\frac{\partial Q_r}{\partial x_r^{a_r}}u_r^{a_r}\Big)\\
&&-\frac{t(k-1)\mathcal{A}^{\frac{1}{k}-2}}{\mathcal{C}}W_s^{b_s}
\sum_{r=1}^nt(k-1)\mathcal{A}^{\frac{1}{k}-2}Q_r^{k-1}\sum_{c_r=1}^{2m_r}W_r^{c_r}\frac{\partial Q_r}{\partial u_r^{c_r}}
\Big(\sum_{l=1}^n\sum_{a_l=1}^{2m_l}Q_l^{k-1}\frac{\partial Q_l}{\partial x_l^{a_l}}u_l^{a_l}\Big)\\
\end{eqnarray*}
Note that for each fixed $r\in\{1,\cdots,n\}$, we have
\begin{eqnarray}
\sum_{c_r=1}^{2m_r}\frac{\partial^2 Q_r}{\partial u_r^{c_r}\partial x_r^{a_r}} u_r^{c_r}&=&2\frac{\partial Q_r}{\partial x_r^{a_r}},\quad\quad\quad
\sum_{c_r=1}^{2m_r}\frac{\partial Q_r}{\partial u_r^{c_r}}u_r^{c_r}=2Q_r.\label{icb}
\end{eqnarray}
Using \eqref{ibl} and \eqref{icb} we have
\begin{eqnarray*}
&&\mathcal{E}_s\sum_{c_s=1}^{2m_s}[B_s]^{b_sc_s}\Bigg(\sum_{a_s=1}^{2m_s}\frac{\partial^2Q_s}{\partial u_s^{c_s}\partial x_s^{a_s}}u_s^{a_s}-\frac{\partial Q_s}{\partial x_s^{c_s}}\Bigg)\\
&=&\sum_{c_s=1}^{2m_s}\Bigg\{
[Q_s]^{b_sc_s}-\frac{t(k-1)\mathcal{A}^{\frac{1}{k}-1}Q_s^{k-2}}{1+t(2k-1)\mathcal{A}^{\frac{1}{k}-1}Q_s^{k-1}}u_s^{b_s}u_s^{c_s}
\Bigg\}\Bigg(\sum_{a_s=1}^{2m_s}\frac{\partial^2Q_s}{\partial u_s^{c_s}\partial x_s^{a_s}}u_s^{a_s}-\frac{\partial Q_s}{\partial x_s^{c_s}}\Bigg)\\
&=&\sum_{c_s=1}^{2m_s}[Q_s]^{b_sc_s}\Bigg(\sum_{a_s=1}^{2m_s}\frac{\partial^2Q_s}{\partial u_s^{c_s}\partial x_s^{a_s}}u_s^{a_s}-\frac{\partial Q_s}{\partial x_s^{c_s}}\Bigg)-\frac{2t(k-1)\mathcal{A}^{\frac{1}{k}-1}Q_s^{k-2}u_s^{b_s}}{1+t(2k-1)\mathcal{A}^{\frac{1}{k}-1}Q_s^{k-1}}\Big(\sum_{a_s=1}^{2m_s}\frac{\partial Q_s}{\partial x_s^{a_s}}u_s^{a_s}\Big)\\
&&+\frac{t(k-1)\mathcal{A}^{\frac{1}{k}-1}Q_s^{k-2}u_s^{b_s}}{1+t(2k-1)\mathcal{A}^{\frac{1}{k}-1}Q_s^{k-1}}\Big(\sum_{c_s=1}^{2m_s}
\frac{\partial Q_s}{\partial x_s^{c_s}}u_s^{c_s}\Big)\\
&=&\sum_{c_s=1}^{2m_s}[Q_s]^{b_sc_s}\Bigg(\sum_{a_s=1}^{2m_s}\frac{\partial^2Q_s}{\partial u_s^{c_s}\partial x_s^{a_s}}u_s^{a_s}-\frac{\partial Q_s}{\partial x_s^{c_s}}\Bigg)-\frac{t(k-1)\mathcal{A}^{\frac{1}{k}-1}Q_s^{k-2}u_s^{b_s}}{1+t(2k-1)\mathcal{A}^{\frac{1}{k}-1}Q_s^{k-1}}\Big(\sum_{a_s=1}^{2m_s}\frac{\partial Q_s}{\partial x_s^{a_s}}u_s^{a_s}\Big).
\end{eqnarray*}
 Note that $[Q_s]^{b_sc_s}\frac{\partial Q_s}{\partial u_s^{c_s}}=u_s^{b_s}$, this together with the second equality in \eqref{icb} imply
 \begin{eqnarray*}
 \sum_{c_s=1}^{2m_s}[B_s]^{b_sc_s}\frac{\partial Q_s}{\partial u_s^{c_s}}=\frac{1}{\mathcal{E}_s}\Bigg\{u_s^{b_s}-\frac{2t(k-1)\mathcal{A}^{\frac{1}{k}-1}Q_s^{k-1}}{1+t(2k-1)\mathcal{A}^{\frac{1}{k}-1}Q_s^{k-1}}u_s^{b_s}\Bigg\}
 =\frac{u_s^{b_s}}{1+t(2k-1)\mathcal{A}^{\frac{1}{k}-1}Q_s^{k-1}}.
 \end{eqnarray*}
 Using \eqref{icb}, we also have
 \begin{eqnarray*}
\sum_{c_r=1}^{2m_r}W_r^{c_r}
\Big(\sum_{a_r=1}^{2m_r}\frac{\partial^2 Q_r}{\partial u_r^{c_r}\partial x_r^{a_r}}u_r^{a_r}-\frac{\partial Q_r}{\partial x_r^{c_r}}\Big)
&=&\frac{Q_r^{k-1}}{1+t(2k-1)\mathcal{A}^{\frac{1}{k}-1}Q_r^{k-1}}\Big(\sum_{a_r=1}^{2m_r}\frac{\partial Q_r}{\partial x_r^{a_r}}u_r^{a_r}\Big),\\
\sum_{c_r=1}^{2m_r}W_r^{c_r}\Big(\sum_{a_r=1}^{2m_r}\frac{\partial Q_r}{\partial u_r^{c_r}}\frac{\partial Q_r}{\partial x_r^{a_r}}u_r^{a_r}\Big)
&=&\frac{2Q_r^{k}}{1+t(2k-1)\mathcal{A}^{\frac{1}{k}-1}Q_r^{k-1}}\Big(\sum_{a_r=1}^{2m_r}\frac{\partial Q_r}{\partial x_r^{a_r}}u_r^{a_r}\Big),\\
\sum_{c_r=1}^{2m_r}W_r^{c_r}\frac{\partial Q_r}{\partial u_r^{c_r}}\Big(\sum_{l=1}^n\sum_{a_l=1}^{2m_l}Q_l^{k-1}\frac{\partial Q_l}{\partial x_l^{a_l}}u_l^{a_l}\Big)
&=&\frac{2Q_r^{k}}{1+t(2k-1)\mathcal{A}^{\frac{1}{k}-1}Q_r^{k-1}}\Big(\sum_{l=1}^n\sum_{a_l=1}^{2m_l}Q_l^{k-1}\frac{\partial Q_l}{\partial x_l^{a_l}}u_l^{a_l}\Big).
\end{eqnarray*}

So that
\begin{eqnarray*}
2\pmb{G}^{b_s}
&=&\sum_{c_s=1}^{2m_s}[Q_s]^{b_sc_s}\Bigg(\sum_{a_s=1}^{2m_s}\frac{\partial^2Q_s}{\partial u_s^{c_s}\partial x_s^{a_s}}u_s^{a_s}-\frac{\partial Q_s}{\partial x_s^{c_s}}\Bigg)\\
&&-\frac{t(k-1)\mathcal{A}^{\frac{1}{k}-1}Q_s^{k-2}u_s^{b_s}}{1+t(2k-1)\mathcal{A}^{\frac{1}{k}-1}Q_s^{k-1}}\sum_{a_s=1}^{2m_s}\frac{\partial^2Q_s}{\partial x_s^{a_s}}u_s^{a_s}\\
&&+\frac{t(k-1)\mathcal{A}^{\frac{1}{k}-1}Q_s^{k-2}u_s^{b_s}}{1+t(2k-1)\mathcal{A}^{\frac{1}{k}-1}Q_s^{k-1}}\Big(\sum_{a_s=1}^{2m_s}\frac{\partial Q_s}{\partial x_s^{a_s}}u_s^{a_s}\Big)\\
&&-\frac{t(k-1)\mathcal{A}^{\frac{1}{k}-2}Q_s^{k-1}u_s^{b_s}}{1+t(2k-1)\mathcal{A}^{\frac{1}{k}-1}Q_s^{k-1}}\Big(\sum_{l=1}^n\sum_{a_l=1}^{2m_l}Q_l^{k-1}\frac{\partial Q_l}{\partial x_l^{a_l}}u_l^{a_l}\Big)
\\
&&+\frac{t(k-1)\mathcal{A}^{\frac{1}{k}-2}}{\mathcal{C}}W_s^{b_s}
\sum_{r=1}^n\frac{\mathcal{E}_rQ_r^{k-1}}{1+t(2k-1)\mathcal{A}^{\frac{1}{k}-1}Q_r^{k-1}}\Big(\sum_{a_r=1}^{2m_r}\frac{\partial Q_r}{\partial x_r^{a_r}}u_r^{a_r}\Big)\\
&&+\frac{t(k-1)\mathcal{A}^{\frac{1}{k}-2}}{\mathcal{C}}W_s^{b_s}
\sum_{r=1}^n\frac{2t(k-1)\mathcal{A}^{\frac{1}{k}-1}Q_r^{2(k-1)}}{1+t(2k-1)\mathcal{A}^{\frac{1}{k}-1}Q_r^{k-1}}\Big(\sum_{a_r=1}^{2m_r}\frac{\partial Q_r}{\partial x_r^{a_r}}u_r^{a_r}\Big)\\
&&-\frac{t(k-1)\mathcal{A}^{\frac{1}{k}-2}}{\mathcal{C}}W_s^{b_s}
\sum_{r=1}^n\frac{2t(k-1)\mathcal{A}^{\frac{1}{k}-2}Q_r^{2k-1}}{1+t(2k-1)\mathcal{A}^{\frac{1}{k}-1}Q_r^{k-1}}\Big(\sum_{l=1}^n\sum_{a_l=1}^{2m_l}Q_l^{k-1}\frac{\partial Q_l}{\partial x_l^{a_l}}u_l^{a_l}\Big).
\end{eqnarray*}

Note that
\begin{eqnarray*}
\frac{\mathcal{E}_rQ_r^{k-1}}{1+t(2k-1)\mathcal{A}^{\frac{1}{k}-1}Q_r^{k-1}}+\frac{2t(k-1)\mathcal{A}^{\frac{1}{k}-1}Q_r^{2(k-1)}}{1+t(2k-1)\mathcal{A}^{\frac{1}{k}-1}Q_r^{k-1}}=Q_r^{k-1},\
\end{eqnarray*}
and
\begin{eqnarray*}
&&1+\frac{1}{\mathcal{C}}\sum_{r=1}^n\frac{2t(k-1)\mathcal{A}^{\frac{1}{k}-2}Q_r^{2k-1}}{1+t(2k-1)\mathcal{A}^{\frac{1}{k}-1}Q_r^{k-1}}\\
&=&\frac{1}{\mathcal{C}}\Bigg\{\frac{1}{\mathcal{A}}\sum_{r=1}^n\frac{(1+t\mathcal{A}^{\frac{1}{k}-1}Q_r^{k-1})Q_r^k}{1+t(2k-1)\mathcal{A}^{\frac{1}{k}-1}Q_r^{k-1}}+\sum_{r=1}^n\frac{2t(k-1)\mathcal{A}^{\frac{1}{k}-2}Q_r^{2k-1}}{1+t(2k-1)\mathcal{A}^{\frac{1}{k}-1}Q_r^{k-1}}\Bigg\}\\
&=&\frac{1}{\mathcal{C\mathcal{A}}}\Bigg\{\sum_{r=1}^n\frac{(1+t\mathcal{A}^{\frac{1}{k}-1}Q_r^{k-1})Q_r^k}{1+t(2k-1)\mathcal{A}^{\frac{1}{k}-1}Q_r^{k-1}}+\sum_{r=1}^n\frac{2t(k-1)\mathcal{A}^{\frac{1}{k}-1}Q_r^{2k-1}}{1+t(2k-1)\mathcal{A}^{\frac{1}{k}-1}Q_r^{k-1}}\Bigg\}\\
&=&\frac{1}{\mathcal{C\mathcal{A}}}\sum_{r=1}^nQ_r^k=\frac{1}{\mathcal{C}}.
\end{eqnarray*}

Now using \eqref{cel} and \eqref{ws}, and rearranging terms we obtain
\begin{eqnarray*}
2\pmb{G}^{b_s}
&=&\sum_{c_s=1}^{2m_s}[Q_s]^{b_sc_s}\Bigg(\sum_{a_s=1}^{2m_s}\frac{\partial^2Q_s}{\partial u_s^{c_s}\partial x_s^{a_s}}u_s^{a_s}-\frac{\partial Q_s}{\partial x_s^{c_s}}\Bigg)\\
&&+\frac{t(k-1)\mathcal{A}^{\frac{1}{k}-2}}{\mathcal{C}}W_s^{b_s}
\Big(\sum_{r=1}^n\sum_{a_r=1}^{2m_r}Q_r^{k-1}\frac{\partial Q_r}{\partial x_r^{a_r}}u_r^{a_r}\Big)\\
&&-\frac{t(k-1)\mathcal{A}^{\frac{1}{k}-2}}{\mathcal{C}}W_s^{b_s}
\Big(\sum_{l=1}^n\sum_{a_l=1}^{2m_l}Q_l^{k-1}\frac{\partial Q_l}{\partial x_l^{a_l}}u_l^{a_l}\Big)\\
&=&\sum_{c_s=1}^{2m_s}[Q_s]^{b_sc_s}\Bigg(\sum_{a_s=1}^{2m_s}\frac{\partial^2Q_s}{\partial u_s^{c_s}\partial x_s^{a_s}}u_s^{a_s}-\frac{\partial Q_s}{\partial x_s^{c_s}}\Bigg)\\
&=&\breve{\varGamma}_{c_s;a_s}^{b_s}(x_s)u_s^{c_s}u_s^{a_s},
\end{eqnarray*}
where in the last equality we used \eqref{rhm} and \eqref{lcc}.
Thus for any fixed $i,j,l\in\{1,\cdots,n\}$ and $a_i\in\{1,\cdots,m_i\},b_j\in\{1,\cdots,m_j\}$
 and $b_l\in\{1,\cdots,m_l\}$,
the real Berwald connection coefficients $\varGamma_{a_i;c_j}^{b_l}$
 associated to $F_{t,k}$ are given by
\begin{eqnarray*}
\varGamma_{a_i;c_j}^{b_l}=\frac{\partial^2 \pmb{G}^{b_l}}{\partial u_i^{a_i}\partial u_j^{c_j}}
=\left\{
   \begin{array}{ll}
     \breve{\varGamma}_{a_l;c_l}^{b_l}, & i=j=l, \\
     0, & \hbox{otherwise},
   \end{array}
 \right.
\end{eqnarray*}
where $\breve{\varGamma}_{a_l;c_l}^{b_l}$ are the Levi-Civita connection coefficients associated to $Q_l$ for $l=1,\cdots,n$.
Thus $F_{t,k}$ is a real Berwald metric, so that the \textbf{S}-curvature of $F_{t,k}$ vanishes identically \cite{Shen}.

\textbf{Step 3.} $F_{t,k}$ is a complex Berwald metric on $M$.

We prove it by showing that the horizontal Chern-Finsler connection coefficients $\varGamma_{\beta;\alpha}^\gamma$
associated to $F_{t,k}$ are independent of $v$. Let's denote
$$
\pmb{H}=\begin{pmatrix}
   \frac{\partial^2G}{\partial v_1^1\partial \overline{v_1^1}} & \cdots & \frac{\partial^2G}{\partial v_1^1\partial \overline{v_1^{m_1}}} & \cdots & \frac{\partial^2G}{\partial v_1^1\partial \overline{v_n^1}} & \cdots & \frac{\partial^2G}{\partial v_1^1\partial \overline{v_n^{m_n}}} \\
    \vdots & \vdots & \vdots & \vdots & \vdots & \vdots &\vdots\\
   \frac{\partial^2G}{\partial v_1^{m_1}\partial \overline{v_1^1}} & \cdots & \frac{\partial^2G}{\partial v_1^{m_1}\partial \overline{v_1^{m_1}}} & \cdots & \frac{\partial^2G}{\partial v_1^{m_1}\partial \overline{v_n^1}} & \cdots & \frac{\partial^2G}{\partial v_1^{m_1}\partial \overline{v_n^{m_n}}} \\
   \vdots & \vdots & \vdots & \vdots & \vdots & \vdots & \vdots \\
   \frac{\partial^2G}{\partial v_n^1\partial \overline{v_1^1}} & \cdots & \frac{\partial^2G}{\partial v_n^1\partial \overline{v_1^{m_1}}} & \cdots & \frac{\partial^2G}{\partial v_n^1\partial \overline{v_n^1}} & \cdots & \frac{\partial^2G}{\partial v_n^1\partial \overline{v_n^{m_n}}} \\
    \vdots & \vdots & \vdots & \vdots & \vdots & \vdots & \vdots \\
   \frac{\partial^2G}{\partial v_n^{m_n}\partial \overline{v_1^1}} & \cdots & \frac{\partial^2G}{\partial v_n^{m_n}\partial \overline{v_1^{m_1}}} & \cdots & \frac{\partial^2G}{\partial v_n^{m_n}\partial \overline{v_n^1}} & \cdots & \frac{\partial^2G}{\partial v_n^{m_n}\partial \overline{v_n^{m_n}}} \\
 \end{pmatrix}.
$$
the complex fundamental tensor matrix of $F_{t,k}$ which is an $N\times N$ matrix, and denote
$$
\quad \pmb{C}=\begin{pmatrix}
                                                                                                    \pmb{C}_1 & \cdots & 0 \\
                                                                                                    \vdots & \ddots & \vdots \\
                                                                                                    0 & \cdots & \pmb{C}_n \\
                                                                                                  \end{pmatrix}
$$
where $\pmb{C}_l=(C_{\alpha_l\overline{\beta_l}})$ is an $m_l\times m_l$ Hermitian matrix with
$$
 C_{\alpha_l\overline{\beta_l}}=\mathcal{E}_l[Q_l]_{\alpha_l\overline{\beta_l}}+t(k-1)\mathcal{A}^{\frac{1}{k}-1}Q_l^{k-2}\frac{\partial Q_l}{\partial v_l^{\alpha_l}}\frac{\partial Q_l}{\partial \overline{v_l^{\beta_l}}}
$$
 for $l=1,\cdots,n$ and $\mathcal{E}_l$ are given by \eqref{bl}.

 For each fixed $l\in \{1,\cdots,n\}$, we denote
\begin{eqnarray*}
\pmb{Y}_l&=&\begin{pmatrix}
        Q_l^{k-1}\frac{\partial Q_l}{\partial v_l^1} \\
        \vdots \\
        Q_l^{k-1}\frac{\partial Q_l}{\partial v_l^{m_l}} \\
      \end{pmatrix},\quad
\pmb{Y}_l^\ast=\Big(Q_l^{k-1}\frac{\partial Q_l}{\partial \overline{v_l^1}},\cdots,Q_l^{k-1}\frac{\partial Q_l}{\partial \overline{v_l^{m_l}}}\Big)
\end{eqnarray*}
and set
\begin{eqnarray*}
\pmb{Y}=\begin{pmatrix}
    \pmb{Y}_1 \\
    \vdots \\
    \pmb{Y}_n \\
  \end{pmatrix}_{N\times 1},\quad
\pmb{Y}^\ast=\begin{pmatrix}
           \pmb{Y}_1^\ast &, \cdots ,& \pmb{Y}_n^\ast \\
         \end{pmatrix}_{1\times N}.
\end{eqnarray*}
Then the $N\times N$ matrix $\pmb{H}$ can be rewritten as
$$
\pmb{H}=\pmb{C}-t(k-1)\mathcal{A}^{\frac{1}{k}-2}\pmb{Y}\pmb{Y}^\ast.
$$

Next we derive the inverse matrix $\pmb{H}^{-1}$ of $\pmb{H}$. First we have
$$
\pmb{C}^{-1}=\begin{pmatrix}
                                                                                                    \pmb{C}_1^{-1} & \cdots & 0 \\
                                                                                                    \vdots & \ddots & \vdots \\
                                                                                                    0 & \cdots & \pmb{C}_n^{-1} \\
                                                                                                  \end{pmatrix},
$$
where $\pmb{C}_l^{-1}=(C^{\overline{\beta_l}\gamma_l})$ denotes the inverse matrix of $\pmb{C}_l$. Using Lemma 6.1 in \cite{Zh1}, we have
\begin{equation}
C^{\overline{\beta_l}\gamma_l}=\mathcal{E}_l^{-1}(v)\Bigg\{[Q_l]^{\overline{\beta_l}\gamma_l}-\frac{t(k-1)\mathcal{A}^{\frac{1}{k}-1}Q_l^{k-2}}{1+tk\mathcal{A}^{\frac{1}{k}-1}Q_l^{k-1}}\overline{v_l^{\beta_l}}v_l^{\gamma_l}
\Bigg\},\quad  l=1,\cdots, n.\label{cbl}
\end{equation}
Note that
\begin{eqnarray}
\mathcal{E}&:=&1-t(k-1)\mathcal{A}^{\frac{1}{k}-2}\pmb{Y}^\ast \pmb{C}^{-1}\pmb{Y}\nonumber\\
&=&1-t(k-1)\mathcal{A}^{\frac{1}{k}-2}\sum_{l=1}^n\frac{Q_l^{2k-1}}{1+t k\mathcal{A}^{\frac{1}{k}-1}Q_l^{k-1}}\nonumber\\
&=&\frac{1}{\mathcal{A}}\sum_{l=1}^n\frac{(1+t\mathcal{A}^{\frac{1}{k}-1}Q_l^{k-1})Q_l^k}{1+tk\mathcal{A}^{\frac{1}{k}-1}Q_l^{k-1}},
\end{eqnarray}
thus using Lemma 6.1 again we have
\begin{eqnarray*}
\pmb{H}^{-1}&=&\pmb{C}^{-1}+\frac{t(k-1)\mathcal{A}^{\frac{1}{k}-2}}{\mathcal{E}}\pmb{C}^{-1}\pmb{Y}\pmb{Y}^\ast \pmb{C}^{-1},
\end{eqnarray*}
or equivalently
\begin{eqnarray*}
H^{\overline{\beta_l}\gamma_i}
&=&C^{\overline{\beta_l}\gamma_l}\delta_{li}+\frac{t(k-1)\mathcal{A}^{\frac{1}{k}-2}}{\mathcal{E}}C^{\overline{\beta_l}\mu_l}Q_l^{k-1}\frac{\partial Q_l}{\partial v_l^{\mu_l}}C^{\overline{\mu_i}\gamma_i}Q_i^{k-1}\frac{\partial Q_i}{\partial\overline{v_i^{\mu_i}}}.
\end{eqnarray*}

Now we have
\begin{eqnarray*}
\frac{\partial^2G}{\partial \overline{v_l^{\beta_l}}\partial z_s^{\alpha_s}}
&=&-t(k-1)\mathcal{A}^{\frac{1}{k}-2}Q_s^{k-1}\frac{\partial Q_s}{\partial z_s^{\alpha_s}}Q_l^{k-1}\frac{\partial Q_l}{\partial \overline{v_l^{\beta_l}}}\\
&&+t(k-1)\mathcal{A}^{\frac{1}{k}-1}Q_l^{k-2}\frac{\partial Q_l}{\partial z_l^{\alpha_l}}\frac{\partial Q_l}{\partial\overline{v_l^{\beta_l}}}\delta_{ls}
+\mathcal{E}_l\frac{\partial^2Q_l}{\partial \overline{v_l^{\beta_l}}\partial z_l^{\alpha_l}}\delta_{ls}.
\end{eqnarray*}
Thus for each fixed $i,s\in\{1,\cdots,n\}$ and $\gamma_i\in\{1,\cdots,m_i\},\alpha_s\in\{1,\cdots,m_s\}$,
the Chern-Finsler nonlinear connection coefficients $\varGamma_{;\alpha_s}^{\gamma_i}$ associated to $F_{t,k}$ are given by
\begin{eqnarray*}
\varGamma_{;\alpha_s}^{\gamma_i}
&=&\sum_{l=1}^n\sum_{\beta_l=1}^{m_l}H^{\overline{\beta_l}\gamma_i}\frac{\partial^2G}{\partial \overline{v_l^{\beta_l}}\partial z_s^{\alpha_s}}\\
&=&-t(k-1)\mathcal{A}^{\frac{1}{k}-2}Q_s^{k-1}Q_i^{k-1}\frac{\partial Q_s}{\partial z_s^{\alpha_s}}\sum_{\beta_i=1}^{m_i}\frac{\partial Q_i}{\partial\overline{v_i^{\beta_i}}}C^{\overline{\beta_i}\gamma_i}\\
&&+t(k-1)\mathcal{A}^{\frac{1}{k}-1}Q_i^{k-2}\frac{\partial Q_i}{\partial z_s^{\alpha_s}}\sum_{\beta_i=1}^{m_i}C^{\overline{\beta_i}\gamma_i}\frac{\partial Q_i}{\partial\overline{v_i^{\beta_i}}}\delta_{si}
+\mathcal{E}_i\sum_{\beta_i=1}^{m_i}\frac{\partial Q_i}{\partial \overline{v_i^{\beta_i}}\partial z_i^{\alpha_i}}C^{\overline{\beta_i}\gamma_i}\delta_{si}\\
&&-\frac{t^2(k-1)^2\mathcal{A}^{\frac{2}{k}-4}}{\mathcal{E}}\Bigg\{\sum_{l=1}^nQ_l^{2(k-1)}\sum_{\beta_l,\mu_l=1}^{m_l}C^{\overline{\beta_l}\mu_l}\frac{\partial Q_l}{\partial v_l^{\mu_l}}\frac{\partial Q_l}{\partial\overline{v_l^{\beta_l}}}\Bigg\}Q_s^{k-1}\frac{\partial Q_s}{\partial z_s^{\alpha_s}}Q_i^{k-1}\sum_{\lambda_i=1}^{m_i}\frac{\partial Q_i}{\partial \overline{v_i^{\lambda_i}}}C^{\overline{\lambda_i}\gamma_i}\\
&&+\frac{t^2(k-1)^2\mathcal{A}^{\frac{2}{k}-3}}{\mathcal{E}}\Bigg\{\sum_{\beta_s,\mu_s=1}^{m_s}C^{\overline{\beta_s}\mu_s}\frac{\partial Q_s}{\partial\overline{v_s^{\beta_s}}}\frac{\partial Q_s}{\partial v_s^{\mu_s}}\Bigg\}Q_s^{2k-3}Q_i^{k-1}\frac{\partial Q_s}{\partial z_s^{\alpha_s}}\sum_{\lambda_i=1}^{m_i}\frac{\partial Q_i}{\partial\overline{v_i^{\lambda_i}}}C^{\overline{\lambda_i}\gamma_i}\\
&&+\frac{t(k-1)\mathcal{A}^{\frac{1}{k}-2}}{\mathcal{E}}\Bigg\{\mathcal{E}_s\sum_{\beta_s,\mu_s=1}^{m_s}C^{\overline{\beta_s}\mu_s}\frac{\partial Q_s}{\partial\overline{ v_s^{\beta_s}}\partial z_s^{\alpha_s}}\frac{\partial Q_s}{\partial v_s^{\mu_s}}\Bigg\}Q_s^{k-1}Q_i^{k-1}\sum_{\lambda_i=1}^{m_i}\frac{\partial Q_i}{\partial \overline{v_i^{\lambda_i}}}C^{\overline{\lambda_i}\gamma_i}.
\end{eqnarray*}
Using \eqref{cbl}, we obtain
\begin{eqnarray}
\sum_{\beta_i=1}^{m_i}\frac{\partial Q_i}{\partial\overline{v_i^{\beta_i}}}C^{\overline{\beta_i}\gamma_i}
&=&\frac{v_i^{\gamma_i}}{1+tk\mathcal{A}^{\frac{1}{k}-1}Q_i^{k-1}},\label{n1}\\
\sum_{\beta_i=1}^{m_i}\mathcal{E}_iC^{\overline{\beta_i}\gamma_i}\frac{\partial^2 Q_i}{\partial \overline{v_i^{\beta_i}}\partial z_i^{\alpha_i}}\delta_{si}
&=&\Bigg\{\hat{\varGamma}_{;\alpha_i}^{\gamma_i}-\frac{t(k-1)\mathcal{A}^{\frac{1}{k}-1}Q_i^{k-2}}{1+t k\mathcal{A}^{\frac{1}{k}-1}Q_i^{k-1}}\frac{\partial Q_i}{\partial z_i^{\alpha_i}}v_i^{\gamma_i}
\Bigg\}\delta_{si},\label{n2}\\
\sum_{l=1}^nQ_l^{2(k-1)}\sum_{\beta_l,\mu_l=1}^{m_l}C^{\overline{\beta_l}\mu_l}\frac{\partial Q_l}{\partial v_l^{\mu_l}}\frac{\partial Q_l}{\partial \overline{v_l^{\beta_l}}}
&=&\sum_{l=1}^n\frac{Q_l^{2k-1}}{1+tk\mathcal{A}^{\frac{1}{k}-1}Q_l^{k-1}},\label{n3}
\end{eqnarray}
where
\begin{equation}
\hat{\varGamma}_{;\alpha_i}^{\gamma_i}=[Q_i]^{\overline{\beta_i}\gamma_i}\frac{\partial^2 Q_i}{\partial \overline{v_i^{\beta_i}}\partial z_i^{\alpha_i}}
=[Q_i]^{\overline{\beta_i}\gamma_i}\frac{\partial^2[Q_i]_{\lambda_i\overline{\beta_i}}(z_i)}{\partial z_i^{\alpha_i}}v_i^{\lambda_i}
\label{hch}
\end{equation}
are the connection coefficients associated to the Hermitian metric $Q_i$.
Substituting \eqref{n1}-\eqref{n3} into the expression of $\varGamma_{;\alpha_s}^{\gamma_i}$, we obtain
\begin{eqnarray*}
\varGamma_{;\alpha_s}^{\gamma_i}
&=&-t(k-1)\mathcal{A}^{\frac{1}{k}-2}\frac{Q_s^{k-1}Q_i^{k-1}v_i^{\gamma_i}}{1+t k\mathcal{A}^{\frac{1}{k}-1}Q_i^{k-1}}\frac{\partial Q_s}{\partial z_s^{\alpha_s}}\\
&&+t(k-1)\mathcal{A}^{\frac{1}{k}-1}\frac{Q_i^{k-2}v_i^{\gamma_i}}{1+t k\mathcal{A}^{\frac{1}{k}-1}Q_i^{k-1}}\frac{\partial Q_i}{\partial z_i^{\alpha_i}}\delta_{si}\\
&&+\Bigg\{\hat{\varGamma}_{;\alpha_i}^{\gamma_i}-\frac{t(k-1)\mathcal{A}^{\frac{1}{k}-1}Q_i^{k-2}v_i^{\gamma_i}}{1+t k\mathcal{A}^{\frac{1}{k}-1}Q_i^{k-1}}\frac{\partial Q_i}{\partial z_i^{\alpha_i}}
\Bigg\}\delta_{si}\\
&&-\frac{t(k-1)\mathcal{A}^{\frac{1}{k}-2}}{\mathcal{E}}(1-\mathcal{E})Q_s^{k-1}\frac{Q_i^{k-1}v_i^{\gamma_i}}{1+t k\mathcal{A}^{\frac{1}{k}-1}Q_i^{k-1}}\frac{\partial Q_s}{\partial z_s^{\alpha_s}}\\
&&+\frac{t^2(k-1)^2\mathcal{A}^{\frac{2}{k}-3}}{\mathcal{E}}\frac{Q_s^{2k-2}}{1+t k\mathcal{A}^{\frac{1}{k}-1}Q_s^{k-1}}\frac{Q_i^{k-1}v_i^{\gamma_i}}{1+t k\mathcal{A}^{\frac{1}{k}-1}Q_i^{k-1}}
\frac{\partial Q_s}{\partial z_s^{\alpha_s}}\\
&&+\frac{t(k-1)\mathcal{A}^{\frac{1}{k}-2}}{\mathcal{E}}\frac{\mathcal{E}_sQ_s^{k-1}}{1+t k\mathcal{A}^{\frac{1}{k}-1}Q_s^{k-1}}\frac{Q_i^{k-1}v_i^{\gamma_i}}{1+t k\mathcal{A}^{\frac{1}{k}-1}Q_i^{k-1}}
\frac{\partial Q_s}{\partial z_s^{\alpha_s}}.
\end{eqnarray*}

Rearranging terms and using \eqref{bl}, we obtain
\begin{eqnarray}
\varGamma_{;\alpha_s}^{\gamma_i}
=\hat{\varGamma}_{;\alpha_s}^{\gamma_i}\delta_{si}=\left\{
                                                                                    \begin{array}{ll}
                                                                                      \hat{\varGamma}_{;\alpha_s}^{\gamma_s}, & \hbox{if}\quad s=i, \\
                                                                                      0, & \hbox{if}\quad s\neq i.
                                                                                    \end{array}
                                                                                  \right.
\label{hcd}
\end{eqnarray}
By \eqref{hch}, $\hat{\varGamma}_{;\alpha_s}^{\gamma_s}$ are complex linear thus holomorphic  with respect to $v_s=(v_s^1,\cdots,v_s^{m_s})$.
So that the horizontal Chern-Finsler connection coefficients $\varGamma_{\beta_l;\alpha_s}^{\gamma_i}$ associated to $F_{t,k}$ satisfy
$$
\varGamma_{\beta_l;\alpha_i}^{\gamma_s}=\frac{\partial}{\partial v_l^{\beta_l}}(\varGamma_{;\alpha_s}^{\gamma_i})\left\{
                                          \begin{array}{ll}
                                            \hat{\varGamma}_{\beta_l;\alpha_l}^{\gamma_l}(z_l), & \hbox{if}\quad l=i=s, \\
                                            0, & \hbox{otherwise},
                                          \end{array}
                                        \right.
$$
where $\hat{\varGamma}_{\beta_l;\alpha_l}^{\gamma_l}(z_l)$ is the Hermitian connection coefficients of $Q_l$ given by \eqref{hcc} for $l=1,\cdots,n$.
Thus $F_{t,k}$ is a complex Berwald metric.

\textbf{Step 4.} Since $\varGamma_{\beta_l;\alpha_l}^{\gamma_l}=\varGamma_{\alpha_l;\beta_l}^{\gamma_l}$ iff
$\hat{\varGamma}_{\beta_l;\alpha}^{\gamma_l}=\hat{\varGamma}_{\alpha_l;\beta_l}^{\gamma_l}$ for $l=1,\cdots,n$. This implies that $(M,F_{t,k})$ is a K\"ahler-Berwald manifold iff $(M_l,Q_l)$ are K\"ahler manifolds for $l=1,\cdots,n$, iff $(M,Q)$ is a reducible K\"ahler manifold.

\textbf{Setp 5.} Note that $F_{t,k}$ is a strongly convex complex Berwald metric as well as a real Berwald metric, the horizontal Chern-Finsler connection associated to $F_{t,k}$ coincides with the pull-back of the Hermitian connection associated to the usual product metric $F_{0}=Q_1+\cdots +Q_n$, and furthermore $Q_1,\cdots,Q_n$ are complete Hermitian metrics on $M_1,\cdots,M_n$ respectively iff $Q$ is a complete Hermitian metric on $M$. This implies that $(M,F_{t,k})$ is a complete strongly convex complex Berwald manifold as well as a real Berwal metric on $M$.

This completes the proof.
\end{proof}

\subsection{Curvature properties of $F_{t,k}$}

In this section, we investigate the curvature properties of $F_{t,k}$ on $M$.
We denote $K_{t,k}(z,v)$ the holomorphic sectional curvature of $F_{t,k}$ along $(z,v)\in\widetilde{M}$ and $K_l(z_l,v_l)$ the holomorphic sectional curvature of $Q_l$ along $(z_l,v_l)\in\widetilde{M_l}$.

\begin{theorem}\label{mthb} Let $(M,Q)$ be a simply connected complete reducible $C^\infty$ Hermitian manifold (resp. K\"ahler manifold) such that $(M_1,Q_1)\times \cdots\times (M_n,Q_n)$ is the de Rham decomposition of $(M,Q)$ and $F_{t,k}$ is defined by \eqref{ftk}.

 (1) If  $K_l\equiv c>0$ for $l=1,\cdots,n$, then
 \begin{equation}
\frac{(1+t)c}{n+t\sqrt[k]{n}}\leq K_{t,k}(z,v)\leq c,\quad \forall (z,v)\in\widetilde{M}.\label{gc-a}
\end{equation}

(2) If  $K_l\equiv c<0$ for $l=1,\cdots,n$, then
 \begin{equation}
c\leq K_{t,k}(z,v)\leq \frac{(1+t)c}{n+t\sqrt[k]{n}},\quad \forall (z,v)\in\widetilde{M}.\label{gc-b}
\end{equation}

(3) If  $K_l\equiv 0$ for $l=1,\cdots,n$, then
 \begin{equation}
K_{t,k}(z,v)\equiv0,\quad \forall (z,v)\in\widetilde{M}.\label{gc-b}
\end{equation}

(4) If  $K_l\equiv c$ for $l=1,\cdots,n$, then for any fixed point $(z,v)\in\widetilde{M}$ and integer $k\geq 2$ we have
\begin{eqnarray}
\lim_{t\rightarrow 0^+}K_{t,k}(z,v)&=&c\cdot \frac{\displaystyle \sum_{l=1}^nQ_l^2(z_l,v_l)}{\Big(\displaystyle \sum_{l=1}^nQ_l(z_l,v_l)\Big)^2},\\
\lim_{t\rightarrow +\infty}K_{t,k}(z,v)&=&c\cdot \frac{\displaystyle\sum_{l=1}^nQ_l^{k+1}(z_l,v_l)}{\Big(\displaystyle\sum_{l=1}^nQ_l^k(z_l,v_l)\Big)^{\frac{1}{k}+1}}.
\end{eqnarray}
\end{theorem}

\begin{proof}
By the formula of holomorphic sectional curvature of a strongly pseudoconvex complex Finsler metric (see (2.5.11) in p. 100,\cite{ap}), the holomorphic sectional curvature $K_{t,k}$ of $F_{t,k}$ along $(z,v)\in\widetilde{M}$ is given by
\begin{equation}
K_{t,k}(z,v)
=-\frac{2}{G^2}\sum_{s,l,i=1}^n\sum_{\gamma_s=1}^{m_s}\sum_{\alpha_i=1}^{m_i}\sum_{\mu_l=1}^{m_l}G_{\gamma_s}\delta_{\overline{\mu_l}}(\varGamma_{;\alpha_i}^{\gamma_s})v_i^{\alpha_i}\overline{v_l^{\mu_l}},\label{cftk}
\end{equation}
where
$$\delta_{\overline{\mu_l}}=\frac{\partial}{\partial \overline{z_l^{\mu_l}}}-\sum_{r=1}^n\sum_{\lambda_r=1}^{m_r}\overline{\varGamma_{;\mu_l}^{\lambda_r}}\frac{\partial}{\partial\overline{ v_r^{\lambda_r}}}.$$
By \eqref{hch} and \eqref{hcd}, we can simply \eqref{cftk} as follows:
\begin{eqnarray*}
K_{t,k}(z,v)
&=&-\frac{2}{(1+t)G^2}\sum_{l=1}^n(1+t\mathcal{A}^{\frac{1}{k}-1}Q_l^{k-1})\frac{\partial Q_l}{\partial v_l^{\gamma_l}}\frac{\partial}{\partial\overline{z^{\mu_l}}}(\hat{\varGamma}_{;\alpha_l}^{\gamma_l})v_l^{\alpha_l}\overline{v_l^{\mu_l}}.
\end{eqnarray*}
Note that the holomorphic sectional curvature $K_l$ of $Q_l$ along any $(z_l,v_l)\in\widetilde{M_l}$ satisfies
\begin{equation}
K_l(z_l,v_l)=-\frac{2}{Q_l^2}\frac{\partial Q_l}{\partial v_l^{\gamma_l}}\frac{\partial}{\partial\overline{z^{\mu_l}}}(\hat{\varGamma}_{;\alpha_l}^{\gamma_l})v_l^{\alpha_l}\overline{v_l^{\mu_l}}\equiv c,\quad \forall l=1,\cdots,n,
\end{equation}
and we make the convention that  $K_l(z_l,v_l)=0$ whenever $v_l=\textbf{0}$ for $l=1,\cdots,n$. Then we have
\begin{eqnarray}
K_{t,k}(z,v)
&=&\frac{1}{(1+t)G^2}\sum_{l=1}^n(1+t\mathcal{A}^{\frac{1}{k}-1}Q_l^{k-1})Q_l^2K_l(\pmb{\pi}_l(z),(\pmb{\pi}_l)_\ast(v))\nonumber\\
&=&c(1+t)\cdot \frac{\displaystyle\sum_{l=1}^nQ_l^2+t\mathcal{A}^{\frac{1}{k}-1}\displaystyle\sum_{l=1}^nQ_l^{k+1}}{\left[\displaystyle\sum_{l=1}^nQ_l+t\mathcal{A}^{\frac{1}{k}}\right]^2},\label{gc}
\end{eqnarray}
which is a $C^\infty$ function of $t\in[0,+\infty)$ for any $(z,v)\in\widetilde{M}$.

Next we give an estimation of \eqref{gc}.
Note that for each $l=1,\cdots,n$ and any $(z,v)\in \widetilde{M}$, we have $Q_l=Q_l(\pmb{\pi}_l(z),\pmb{\pi}_\ast(v))=Q_l(z_l,v_l)\geq 0$ and $\mathcal{A}>0$,  so that an elementary observation yields
$$
\frac{\displaystyle\sum_{l=1}^nQ_l^2+t\mathcal{A}^{\frac{1}{k}-1}\displaystyle\sum_{l=1}^nQ_l^{k+1}}{\left[\displaystyle\sum_{l=1}^nQ_l+t\mathcal{A}^{\frac{1}{k}}\right]^2}\leq \frac{1}{1+t}
$$
and
$$
(n+\sqrt[k]{n})\left\{\sum_{l=1}^nQ_l^2+t\mathcal{A}^{\frac{1}{k}-1}\sum_{l=1}^nQ_l^{k+1}\right\}-\left[\sum_{l=1}^nQ_l+t\mathcal{A}^{\frac{1}{k}}\right]^2\geq 0.
$$
Thus we immediately obtain \eqref{gc-a}-\eqref{gc-b}.

Finally, for any fixed $(z,v)\widetilde{M}$ and integer $k\geq 2$, it follows from \eqref{gc} that
$$
\lim_{t\rightarrow +\infty}K_{t,k}(z,v)=c\cdot \frac{\displaystyle\sum_{l=1}^nQ_l^{k+1}(z_l,v_l)}{\Big(\displaystyle\sum_{l=1}^nQ_l^k(z_l,v_l)\Big)^{\frac{1}{k}+1}}.
$$

\end{proof}

\begin{remark}
Under the assumption that the holomorphic sectional curvature $K_l \equiv c$ for $l=1,\cdots,n$, the limit
$$c\cdot \frac{\displaystyle \sum_{l=1}^nQ_l^2(z_l,v_l)}{\Big(\displaystyle \sum_{l=1}^nQ_l(z_l,v_l)\Big)^2}$$
is exactly the holomorphic sectional curvature of the usual product metric $F_0^2=Q_1+\cdots+Q_n$ on $M$.
\end{remark}

\begin{remark}By the proof of Theorem \ref{mthb}, it also follows that
if the holomorphic sectional curvature $K_l$ of $Q_l$ is bounded above by a negative constant $c_l$ for $l=1,\cdots,n$, then
$$K_{t,k}\leq \frac{(1+t)c}{n+\sqrt[k]{n}}\quad \mbox{with}\quad c=\min\{c_1,\cdots,c_n\};$$
and if the holomorphic sectional curvature $K_l$ of $Q_l$ is bounded below by a positive constant $c_l$ for $l=1,\cdots,n$, then
$$K_{t,k}\geq  \frac{(1+t)c}{n+\sqrt[k]{n}} \quad \mbox{with}\quad c=\max\{c_1,\cdots,c_n\}.$$

\end{remark}

\section{Holomorphic invariant complex Finsler metrics}

\subsection{The isometry group of $(M,F_{t,k})$}

 Let $M$ be a connected complex manifold, and $\mbox{Aut}(M)$ the group of holomorphic automorphism of $M$, and $\mbox{Aut}^\circ(M)$ the connected component of the identity in $\mbox{Aut}(M)$. We equip $\mbox{Aut}(M)$ with the compact-open topology which is  the topology with neighborhood basis given by all sets of the form $\{f\in\mbox{Aut}(M):f(K)\subset U\}$, where $K\subset M$ is compact and $U\subset M$ is open. It is well-known that $\mbox{Aut}(M)$ is a topological group with respect to this topology. We denote $\mbox{Aut}_p(M)$ the isotropy group of $\mbox{Aut}(M)$ at a point $p\in M$, and $\mbox{Aut}_p^\circ(M)$ the connected components of the identity in $\mbox{Aut}(M)$.
If $G$ is a finite dimensional Lie group
of $\mbox{Aut}(M)$ and there is a continuous homomorphism $\rho: G\rightarrow \mbox{Aut}(M)$
$$
G\times M\ni (g,p)\mapsto (\rho(g))(p)\in M,
$$
then we say that $G$ acts on $M$ as a Lie transformation group through $\rho$. If $G$ acts transitively on $M$ then $M$ is called homogeneous.

\begin{definition}Let $M$ be a connected complex manifold endowed with a strongly pseudoconvex complex Finsler metric $F$.
 A holomorphic transformation $f$ on $M$ is called an isometry of $(M,F)$ if
 \begin{equation}
 F(f(z),f_\ast(v))=F(z,v),\quad \forall z\in M,\forall (z,v)\in T^{1,0}M.\label{hfm}
 \end{equation}
 The isometry group of $(M,F)$ is denote by $I(M,F)$ and the connected component of the identity is denoted by $I^\circ (M,F)$.
 \end{definition}

\begin{remark}
A holomorphic transformation $f$ on $M$ satisfying \eqref{hfm} is necessary in $\mbox{Aut}(M)$, thus we have $I(M,F)\subset\mbox{Aut}(M)$. It is well-known that if $F$ is the Bergman metric on $M$, then $F$ is $\mbox{Aut}(M)$-invariant, and $I(M,F)=\mbox{Aut}(M)$. For a general Hermitian manifold $(M,F)$, $I(M,F)$ is a Lie group. For a strongly convex complex Finsler manifold $(M,F)$ where $F$ is non-Hermitian quadratic, $I(M,F)$ is also a Lie group (cf. Theorem 3.4 in \cite{Deng}).
\end{remark}

 A complex Finsler manifold $(M,F)$ is called homogeneous if $I(M,F)$ acts transitively on $M$. It is clear that if $M$ is connected, then the connected component  $I^\circ(M,F)$ also acts transitively on $M$.


\begin{theorem}\label{mthc}
Let $(M,Q)$ be a simply connected complete reducible $C^\infty$ Hermitian manifold (resp. K\"ahler manifold) such that
$(M,Q)=(M_1,Q_1)\times \cdots\times (M_n,Q_n)$ is the de Rham decomposition of $(M,Q)$. Let $F_{t,k}$ be defined by \eqref{ftk}.
Then

(1) $I^\circ(M,F_{t,k})=I^\circ (M,Q)\cong I^\circ (M_1,Q_1)\times \cdots \times I^\circ (M_n,Q_n)$ for any $t\in[0,+\infty)$ and integer $k\geq 2$;

(2) $(M,F_{t_1,k_1})$ is not holomorphic isometry to $(M,F_{t_2,k_2})$ for any $t_1,t_2\in [0,+\infty),t_1\neq t_2$ and integers $k_1,k_2\geq 2,k_1\neq k_2$.
\end{theorem}
\begin{proof} By Theorem 1 in \cite{Hano}, we have  $I^\circ (M,Q)\cong I^\circ(M_1,Q_1)\times\cdots\times I^\circ (M_n,Q_n)$. Next it is easy to check that
the the natural projection $\pmb{\pi}_l: (M,F_{t,k})\rightarrow (M_l,Q_l)$ and embedding $\pmb{i}_l:(M_l,Q_l)\rightarrow (M,F_{t,k})$ are holomorphic isometries for any $l=1,\cdots,n$.
Thus for each $f=(f_1,\cdots,f_n)\in I^\circ(M,F_{t,k})$ and $(z,v)\in\widetilde{M}$, we have
$$
F(f(z),f_\ast (v))=Q_l(\pmb{\pi}_l(f(z)),(\pmb{\pi}_l)_\ast (f_\ast(v))),
$$
which implies that $f_l=\pmb{\pi}_l\circ f:(M,F_{t,k})\rightarrow (M_l,Q_l)$ are holomorphic isometries for any $l=1,\cdots,n$.
So that for each $f\in I^\circ(M,F_{t,k})$, we have a unique $\tilde{f}:=(\tilde{f}_l,\cdots,\tilde{f}_n)$ with $\tilde{f}_l=\pmb{\pi}_l\circ f\circ \pmb{i}_l\in I^\circ(M_l,Q_l)$ for $l=1,\cdots,n$.

Conversely, for each $g_l\in I^\circ(M_l,Q_l)$ for $l=1,\cdots,n$ we set $g=(g_1,\cdots,g_n)$, which is defined in a natural way on $M$
and it is clear that $g\in I^\circ(M,F_{t,k})$. Thus we have established an one-to-one correspondence between the set $I^\circ(M,F_{t,k})$ and
$I^\circ(M_1,Q_1)\times \cdots\times I^\circ(M_n,Q_n)$. The assertion (1) follows.

For the assertion (2), it suffice to consider the identity map $\mbox{id}:M\rightarrow M$. Then $F_{t_1,k_1}(z,v)=F_{t_2,k_2}(z,v)$ for any $(z,v)\in T^{1,0}M$ iff $t_1=t_2$ and $k_1=k_2$.
\end{proof}

\subsection{Holomorphic invariant complex Finsler metrics on the polydisks}

It is well-known that intrinsic metrics on complex manifolds $M$ such as the Carath$\acute{\mbox{e}}$odory pseudometric, the Kobayashi pseudometric \cite{Ko}, the Sibony metric \cite{Si} and the Bergman metric, are all $\mbox{Aut}(M)$-invariant. They are all complex Finsler metrics,  besides the Bergman metric however, they are in general not smooth. It is an interesting question to ask whether there are $\mbox{Aut}(M)$-invariant strongly pseudoconvex complex Finsler metrics which are not Hermitian quadratic? The existence of nontrivial examples may possibly enhance the study of holomorphic invariant metrics on complex manifold from the view point of differential geometry.

\begin{theorem}\label{phi}
Let $P_n$ be the unit polydisk in $\mathbb{C}^n$ with $n\geq 2$. For any fixed $t\in[0,+\infty)$ and integer $k\geq 2$, we define
\begin{eqnarray}
F_{t,k}(z,v)=\frac{1}{\sqrt{1+t}}\sqrt{\sum_{l=1}^n\frac{|v^l|^2}{(1-|z^l|^2)^2}+t\sqrt[k]{\sum_{l=1}^n\frac{|v^l|^{2k}}{(1-|z^l|^2)^{2k}}}},\quad \forall (z,v)\in T^{1,0}P_n\label{pm}
\end{eqnarray}
or
equivalently
\begin{equation}
F_{t,k}(z,dz)=\frac{1}{\sqrt{1+t}}\sqrt{\sum_{l=1}^n\frac{|dz^l|^2}{(1-|z^l|^2)^2}+t\sqrt[k]{\sum_{l=1}^n\frac{|dz^l|^{2k}}{(1-|z^l|^2)^{2k}}}},\quad \forall z\in P_n.\label{dpm}
\end{equation}
Then

(1) $F_{t,k}$ is an $\mbox{Aut}(P_n)$-invariant complete strongly convex K\"ahler-Berwald metric  with holomorphic sectional curvatures $\in\left[-4, -\frac{4(1+t)}{n+t\sqrt[k]{n}}\right]$;

(2) $F_{t,k}$ has the same geodesic (as sets) as that of the Bergman metric on $P_n$ and the geodesic distance of $Z_1=(z_1^1,\cdots,z_1^n),Z_2=(z_2^1,\cdots,z_2^n)\in P_n$  with respect to $F_{t,k}$ is given by
\begin{eqnarray*}
\sigma(Z_1,Z_2)&=&
\frac{1}{2\sqrt{1+t}}\sqrt{\sum_{l=1}^n\log^2\Bigg(\frac{1+\Big|\frac{z_2^l-z_1^l}{1-\overline{z_1^l}z_2^l}\Big|}{1-\Big|\frac{z_2^l-z_1^l}{1-\overline{z_1^l}z_2^l}\Big|}\Bigg)
+t\sqrt[k]{\sum_{l=1}^n\log^{2k} \Bigg(\frac{1+\Big|\frac{z_2^l-z_1^l}{1-\overline{z_1^l}z_2^l}\Big|}{1-\Big|\frac{z_2^l-z_1^l}{1-\overline{z_1^l}z_2^l}\Big|}\Bigg)}}.
\end{eqnarray*}
\end{theorem}

\begin{proof}
(1) Note that $P_n=M_1\times \cdots\times M_n$ is the product manifold of $n$ copies of the unit disk $M_l=\{z^l\in\mathbb{C}||z^l|<1\}$ equipping with the Poincar$\acute{\mbox{e}}$ metric $Q_l(z^l,dz^l)=\frac{|dz^l|^2}{(1-|z^l|^2)^2}$ for $l=1,\cdots,n$ with Gauss curvature $-4$. By Theorem \ref{mtha} and \ref{mthb}, it suffice to prove that $F_{t,k}$ is $\mbox{Aut}(P_n)$-invariant.
That is, $F_{t,k}(w,dw)=F_{t,k}(z,dz)$ for any $w=f(z)=(f_1(z),\cdots,f_n(z))\in \mbox{Aut}(P_n)$.  Note that according to Proposition $3$ of  Chapter $5$ in \cite{N}, every $w=f(z)\in\mbox{Aut}(P_n)$ must be of the form:
$$
f(z)=\Big(e^{i\theta_1}\frac{z^{\sigma(1)}-a_1}{1-\overline{a_1}z^{\sigma(1)}},\cdots,e^{i\theta_n}\frac{z^{\sigma(n)}-a_n}{1-\overline{a_n}z^{\sigma(n)}}\Big),
$$
where $\theta_1,\cdots,\theta_n\in \mathbb{R}$, $a=(a_1,\cdots,a_n)\in P_n$ and $\sigma:\{1,\cdots,n\}\rightarrow\{1,\cdots,n\}$ is any permutation. That is,
$$
f_l(z)=e^{i\theta_l}\frac{z^{\sigma(l)}-a_l}{1-\overline{a_l}z^{\sigma(l)}},\quad l=1,\cdots,n.
$$

Now we have
\begin{eqnarray*}
F_{t,k}(w,dw)&=&\frac{1}{\sqrt{1+t}}\sqrt{\sum_{l=1}^n\frac{|dw^l|^2}{(1-|w^l|^2)^2}+t\sqrt[k]{\sum_{l=1}^n\frac{|dw^l|^{2k}}{(1-|w^l|^2)^{2k}}}}\\
&=&\frac{1}{\sqrt{1+t}}\sqrt{\sum_{l=1}^n\frac{|df_l(z)|^2}{(1-|f_l(z)|^2)^2}+t\sqrt[k]{\sum_{l=1}^n\frac{|df_l(z)|^{2k}}{(1-|f_l(z)|^2)^{2k}}}}\\
&=&\frac{1}{\sqrt{1+t}}\sqrt{\sum_{l=1}^n\frac{|dz^{\sigma(l)}|^2}{(1-|dz^{\sigma(l)}|^2)^2}+t\sqrt[k]{\sum_{l=1}^n\frac{|dz^{\sigma(l)}|^{2k}}{(1-|z^{\sigma(l)}|^2)^{2k}}}}\\
&=&\frac{1}{\sqrt{1+t}}\sqrt{\sum_{l=1}^n\frac{|dz^l|^2}{(1-|z^l|^2)^2}+t\sqrt[k]{\sum_{l=1}^n\frac{|dz^l|^{2k}}{(1-|z^l|^2)^{2k}}}}\\
&=&F_{t,k}(z,dz),
\end{eqnarray*}
 which completes the assertion (1).

(2) Note that $F_{t,k}$ and the Bergman metric $F_{0}^2=Q_1+\cdots+Q_n$ on $P_n$ obey the same system of geodesic equations, i.e.,
$$
\frac{d^2z^l(s)}{ds^2}+\frac{2\overline{z^l(s)}}{1-|z^l(s)|^2}\frac{dz^l(s)}{ds}\frac{dz^l(s)}{ds}=0,\quad l=1,\cdots,n
$$
and
\begin{eqnarray*}
z^l(s)&=&\frac{e^{a_ls}-1}{e^{a_ls}+1},\\
a_l&=&\frac{1}{s_0}\log \frac{1+r_0^l}{1-r_0^l},\quad s\in\mathbb{R}
\end{eqnarray*}
is a geodesic connecting $0$ and $z_0=(r_0^1,\cdots,r_0^n)$ in $P_n$ such that $z^l(0)=0$ and $z^l(s_0)=r_0^l$ for $l=1,\cdots,n$ (cf. \cite{Look1,Look2}). Thus the geodesic distance between $0$ and $z_0$ with respect to $F_{t,k}$ is given by
\begin{eqnarray*}
\sigma(0,z_0)&=&\int_0^{s_0}F\Big(z^1(s),\cdots,z^n(s),\frac{dz^1(s)}{ds},\cdots,\frac{dz^n(s)}{ds}\Big)ds.
\end{eqnarray*}
Note that
\begin{eqnarray*}
\sum_{l=1}^n\Bigg(\frac{\frac{d}{ds}|z^l(s)|}{1-|z^l(s)|^2}\Bigg)^{2k}&=&\sum_{l=1}^n\Bigg(\frac{1}{2}\frac{d}{ds}\log \frac{1+|z^l(s)|}{1-|z^l(s)|}\Bigg)^{2k}\\
&=&\sum_{l=1}^n\Bigg(\frac{1}{2}\frac{d}{ds}\log \frac{1+\Big(\frac{e^{a_ls}-1}{e^{a_ls}+1}\Big)}{1-\Big(\frac{e^{a_ls}-1}{e^{a_ls}+1}\Big)}\Bigg)^{2k}\\
&=&\frac{1}{2^{2k}}\sum_{l=1}^na_l^{2k}.
\end{eqnarray*}
Thus
\begin{eqnarray*}
\sigma(0,z_0)&=&\frac{1}{2\sqrt{1+t}}\int_0^{s_0}\sqrt{\sum_{l=1}^na_l^2+t\sqrt[k]{\sum_{l=1}^na_l^{2k}}}ds\\
&=&\frac{1}{2\sqrt{1+t}}\sqrt{\sum_{l=1}^na_l^2+t\sqrt[k]{\sum_{l=1}^na_l^{2k}}}s_0\\
&=&\frac{1}{2\sqrt{1+t}}\sqrt{\sum_{l=1}^n\log^2\Big(\frac{1+r_0^l}{1-r_0^l}\Big)+t\sqrt[k]{\sum_{l=1}^n\log^{2k} \Big(\frac{1+r_0^l}{1-r_0^l}\Big)}}.
\end{eqnarray*}
Now let $Z_1=(z_1^1,\cdots,z_1^n),Z_2=(z_2^1,\cdots,z_2^n)\in \triangle^n$ and
$$W(z)=\Big(e^{i\theta_1}\frac{z^1-z_1^1}{1-\overline{z_1^1}z^1},\cdots,e^{i\theta_n}\frac{z^n-z_1^n}{1-\overline{z_1^n}z^n}\Big)\in\mbox{Aut}(\triangle^n)$$ such that  $W(Z_1)=0$ and $W(Z_2)=z_0$
 then
\begin{eqnarray*}
\sigma(Z_1,Z_2)&=&
\frac{1}{2\sqrt{1+t}}\sqrt{\sum_{l=1}^n\log^2\Bigg(\frac{1+\Big|\frac{z_2^l-z_1^l}{1-\overline{z_1^l}z_2^l}\Big|}{1-\Big|\frac{z_2^l-z_1^l}{1-\overline{z_1^l}z_2^l}\Big|}\Bigg)
+t\sqrt[k]{\sum_{l=1}^n\log^{2k} \Bigg(\frac{1+\Big|\frac{z_2^l-z_1^l}{1-\overline{z_1^l}z_2^l}\Big|}{1-\Big|\frac{z_2^l-z_1^l}{1-\overline{z_1^l}z_2^l}\Big|}\Bigg)}}.
\end{eqnarray*}
This completes the proof.
\end{proof}

\begin{proposition}[\cite{Cartan}]\label{C}
Let $D=D_1\times D_2, D_j\subset\mathbb{C}^{n_j},n=n_1+n_2,n_1,n_2>0,$ be a bounded domain. Then any $f\in\mbox{Aut}(D)$ which belongs to the connected component of the identity in $\mbox{Aut}(D)$ is of the form
$$f=f_1\times f_2,$$
where $f_j\in\mbox{Aut}(D_j)$.

\end{proposition}

By Theorem \ref{nb}, the unit ball $B_n$ in $\mathbb{C}^n$ does not admit any $\mbox{Aut}(B_n)$-invariant strongly pseudoconvex complex Finsler metric other
 than a constant multiple of the Poincar$\acute{\mbox{e}}$-Bergman metric. The product of open unit balls $\pmb{B}$, however, admits infinite many $\mbox{Aut}(\pmb{B})$-invariant strongly convex K\"ahler-Berwald metrics which are not necessary Hermitian quadratic. More precisely we have
\begin{proposition}\label{prb}
Let $\pmb{B}=B_{1}\times \cdots\times B_{n}\subset \mathbb{C}^N$ with $N=m_1+\cdots +m_n$ be the product of unit balls $B_l\subset\mathbb{C}^{m_l}$
 endowed with the Poincar$\acute{\mbox{e}}$-Bergman metrics
\begin{eqnarray*}
Q_l(z_l,v_l)=\frac{\displaystyle\sum_{\alpha_l=1}^{m_l}|v_l^{\alpha_l}|^2}{1-\Big(\displaystyle\sum_{\alpha_l=1}^{m_l}|z_l^{\alpha_l}|^2\Big)}+\frac{\displaystyle\sum_{\alpha_l=1}^{m_l}z_l^{\alpha_l}\overline{v_l^{\alpha_l}}}{\Big[1-\Big(\displaystyle\sum_{\alpha_l=1}^{m_l}|z_l^{\alpha_l}|^2\Big)\Big]^2},
\quad \forall v_l\in T_{z_l}^{1,0}B_l
\end{eqnarray*}
for $l=1,\cdots,n$.
Then for any fixed $t\in[0,+\infty)$ and integer $k\geq 2$,
\begin{eqnarray}
F_{t,k}(z,v)=\frac{1}{\sqrt{1+t}}\sqrt{\sum_{l=1}^nQ_l(z_l,v_l))+t\sqrt[k]{\sum_{l=1}^nQ_l^k(z_l,v_l)}},\quad \forall (z,v)\in T^{1,0}\pmb{B}\label{pbm}
\end{eqnarray}
is a complete strongly convex K\"ahler-Berwald metric with holomorphic sectional curvatures $\in\left[-4,-\frac{4(1+t)}{n+t\sqrt[k]{n}}\right]$.
\end{proposition}
\begin{proof}
Since the Poincar$\acute{\mbox{e}}$-Bergman metric $Q_l$ is an $\mbox{Aut}(B_l)$-invariant complete K\"ahler metric on $B_l$ with constant holomorphic sectional curvature $-4$ for $l=1,\cdots,n$,
 the assertion follows from Theorem \ref{mtha}, Theorem \ref{mthb} and Proposition \ref{C}.
\end{proof}

\begin{remark}
 By Proposition \ref{C}, we have
$$
I^\circ (\pmb{B},F_{t,k})\cong I^\circ(B_1,Q_1)\times \cdots \times I^\circ(B_n,Q_n)=\mbox{Aut}^\circ(B_1)\times \cdots\times\mbox{Aut}^\circ(B_n),
$$
since $Q_l$ is an $\mbox{Aut}(B_l)$-invariant Bergman metric for any $l=1,\cdots,n$.
Note that $I^\circ (\pmb{B},F_{t,k})$ acts transitively on $\pmb{B}$ since $\mbox{Aut}^\circ(B_1)\times\cdots\times \mbox{Aut}^\circ(B_n)$ acts transitively on $\pmb{B}$.
\end{remark}

\subsection{Complex Finsler metrics on biholomorphically equivalent manifolds}

The following theorem shows that holomorphic invariant strongly pseudoconvex complex Finsler metrics behave very well from the view point of both complex analysis and differential geometry.
\begin{theorem}\label{bh}
Let $D_1$ and $D_2$ be two domains in $\mathbb{C}^n$ and $D_1$ is biholomorphically equivalent to $D_2$. Then $D_2$ admits an $\mbox{Aut}(D_2)$-invariant strongly pseudoconvex complex Finsler metric iff $D_1$ admits an $\mbox{Aut}(D_1)$-invariant strongly pseudoconvex complex Finsler metric. More precisely, the following assertions hold:

(1) $D_2$ admits an $\mbox{Aut}(D_2)$-invariant Hermitian metric iff $D_1$ admits an $\mbox{Aut}(D_1)$-invariant Hermitian metric;

(2) $D_2$ admits an non-Hermitian quadratic $\mbox{Aut}(D_2)$-invariant strongly pseudoconvex complex Finsler metric iff $D_1$ admits an non-Hermitian quadratic $\mbox{Aut}(D_1)$-invariant strongly pseudoconvex complex Finsler metric;

(3) $D_2$ admits an $\mbox{Aut}(D_2)$-invariant complex Berwald metric (resp. weakly complex Berwald metric) iff $D_1$ admits an $\mbox{Aut}(D_1)$-invariant complex Berwald metric (resp.  weakly complex Berwald metric);

(4) $D_2$ admits an $\mbox{Aut}(D_2)$-invariant K\"ahler-Finsler metric (resp. weakly K\"ahler-Finsler metric)  iff $D_1$ admits an $\mbox{Aut}(D_1)$-invariant K\"ahler-Finsler metric (resp. weakly K\"ahler-Finsler metric).
\end{theorem}

\begin{proof} Denote $z_1=(z_1^1,\cdots,z_1^n)$ and $z_2=(z_2^1,\cdots,z_2^n)$ the holomorphic coordinates on $D_1$ and $D_2$ respectively so that $(z_1,v_1)=(z_1^1,\cdots,z_1^n, v_1^1,\cdots,v_1^n)$ and
$(z_2,v_2)=(z_2^1,\cdots,z_2^n,v_2^1,\cdots,v_2^n)$ are holomorphic coordinates on $T^{1,0}D_1\cong D_1\times \mathbb{C}^n$ and $T^{1,0}D_2\cong D_2\times \mathbb{C}^n$, respectively. Since $D_1$ is biholomorphically equivalent to $D_2$, there exists holomorphic maps $f=(f_1,\cdots,f_n):D_1\rightarrow D_2$ and $g=(g_1,\cdots,g_n):D_2\rightarrow D_1$ such that
\begin{equation}
g\circ f=\mbox{id}_{D_1}\quad \mbox{and}\quad f\circ g=\mbox{id}_{D_2}.\label{fg}
\end{equation}

Suppose that $F_1: T^{1,0}D_1\rightarrow [0,+\infty)$ is an $\mbox{Aut}(D_1)$-invariant strongly pseudoconvex complex Finsler metric on $D_1$. Define
\begin{equation}
F_2(z_2,v_2):=F_1(g(z_2),(g_\ast)_{z_2}(v_2)),\quad \forall v_2\in T_{z_2}^{1,0}D_2.\label{f2}
\end{equation}
It is suffice to show that $F_2$ is an $\mbox{Aut}(D_2)$-invariant strongly pseudoconvex complex Finsler metric.

Denote $G_1=F_1^2$ and $G_2=F_2^2$. Then by \eqref{f2}, we have
\begin{eqnarray}
\frac{\partial^2G_2}{\partial v_2^\alpha\partial\overline{v_2^\beta}}&=&\frac{\partial^2G_1}{\partial v_1^\mu\partial \overline{v_1^\nu}}\frac{\partial g_\mu}{\partial z_2^\alpha}\overline{\frac{\partial g_\nu}{\partial z_2^\beta}}\label{sp}\\
\frac{\partial^3G_2}{\partial v_2^\alpha\partial\overline{v_2^\beta}\partial v_2^\gamma}&=&\frac{\partial^3G_1}{\partial v_1^\mu\partial \overline{v_1^\nu}\partial v_1^\lambda}\frac{\partial g_\mu}{\partial z_2^\alpha}\overline{\frac{\partial g_\nu}{\partial z_2^\beta}}\frac{\partial g_\lambda}{\partial z_2^\lambda}.\label{hm}
\end{eqnarray}
Since the $n\times n$ Jacobi matrix $(\frac{\partial g_\mu}{\partial z_2^\alpha})$ is nonsingular and invertible,
 \eqref{sp} implies that
$$\Big(\frac{\partial^2G_2}{\partial v_2^\alpha\partial\overline{v_2^\beta}}\Big)_{n\times n}\quad\mbox{is positive definite iff} \quad \Big(\frac{\partial^2G_1}{\partial v_1^\mu\partial \overline{v_1^\nu}}\Big)\quad
 \mbox{is positive definite}.$$
This implies that $F_2$ is a strongly pseudoconvex complex Finsler metric on $D_2$.

Now \eqref{hm} implies that the complex Cartan tensors
$$\frac{\partial^3G_2}{\partial v_2^\alpha\partial\overline{v_2^\beta}\partial v_2^\gamma}=0\quad \mbox{iff}\quad \frac{\partial^3G_1}{\partial v_1^\mu\partial\overline{v_1^\nu}\partial v_1^\lambda}=0,$$
and
$$
\frac{\partial^3G_2}{\partial v_2^\alpha\partial\overline{v_2^\beta}\partial v_2^\gamma}\neq 0\quad \mbox{iff}\quad \frac{\partial^3G_1}{\partial v_1^\mu\partial\overline{v_1^\nu}\partial v_1^\lambda}\neq 0,
$$
which in return implies that $F_2$ is a Hermitian quadratic metric iff $F_1$ is a Hermitian quadratic metric, and $F_2$ is a non-Hermitian quadratic metric iff $F_1$ is a non-Hermitian quadratic metric.

Next we check that $F_2$ is $\mbox{Aut}(D_2)$-invariant. For any given $h\in\mbox{Aut}(D_2)$, we have $g\circ h\circ f\in \mbox{Aut}(D_1)$. Using the fact that $F_1$ is $\mbox{Aut}(D_1)$-invariant, we have
\begin{eqnarray*}
F_2(h(z_2), h_\ast(v_2))
&=& F_1(g(h(z_2)),g_\ast(h_\ast(v_2)))\\
&=&F_1(g\circ h\circ f)(g(z_2)),(g\circ h\circ f)_\ast(g_\ast(v_2)))\\
&=&F_1(g(z_2),g_\ast(v_2))\\
&=&F_2(z_2, v_2),\quad \forall v_2\in T_{z_2}^{1,0}D_2,
\end{eqnarray*}
where in the second equality we use \eqref{fg} and in the last equality we use \eqref{f2}. Thus $F_2$ is $\mbox{Aut}(D_2)$-invariant. This completes the assertions (1) and (2).

For the assertion (3), let's denote $(G_2^{\overline{\beta}\gamma})$ and $(G_1^{\overline{\nu}\sigma})$ the inverse matrix of $(\frac{\partial^2G_2}{\partial v_2^\alpha\partial\overline{v_2^\beta}})$ and $(\frac{\partial^2G_1}{\partial v_1^\mu\partial \overline{v_1^\nu}})$, respectively. Denote $\breve{\varGamma}_{;\alpha}^\gamma$ and
$\breve{\varGamma}_{\beta;\alpha}^{\gamma}$ the Chern-Finsler nonlinear connection coefficients and horizontal Chern-Finsler connection coefficients associated to $F_2$, respectively;  $\hat{\varGamma}_{;\lambda}^\sigma$ and $\hat{\varGamma}_{\mu;\lambda}^{\sigma}$ the Cher-Finsler nonlinear connection coefficients  and the horizontal Chern-Finsler connection coefficients associated to $F_1$, respectively. Then by \eqref{sp} and \eqref{f2} and the facts that $(\frac{\partial f_\gamma}{\partial z_1^\sigma})$ and $(\frac{\partial g_\lambda}{\partial z_2^\alpha})$ are holomorphic and invertible $n$-by-$n$ matrices, we have
\begin{eqnarray*}
G_2^{\overline{\beta}\gamma}&=&G_1^{\overline{\nu}\sigma}\frac{\partial f_\gamma}{\partial z_1^\sigma}\overline{\frac{\partial f_\beta}{\partial z_1^\nu}}, \\
\frac{\partial^2G_2}{\partial \overline{v_2^\tau}\partial z_2^\alpha}&=&\Bigg\{\frac{\partial^2G_1}{\partial \overline{v_1^\mu}\partial z_1^\lambda}\frac{\partial g_\lambda}{\partial z_2^\alpha}+\frac{\partial^2G_1}{\partial v_1^\beta\partial \overline{v_1^\mu}}\frac{\partial^2g_\beta}{\partial z_2^\lambda\partial z_2^\alpha}v_2^\lambda\Bigg\}\overline{\frac{\partial g_\mu}{\partial z_2^\tau}}.
\end{eqnarray*}
Thus
\begin{eqnarray}
\breve{\varGamma}_{;\alpha}^{\gamma}=G_2^{\overline{\tau}\gamma}\frac{\partial^2G_2}{\partial\overline{v_2^\tau}\partial z_2^\alpha}=\hat{\varGamma}_{;\lambda}^{\sigma}\frac{\partial f_\gamma}{\partial z_1^\sigma}\frac{\partial g_\lambda}{\partial z_2^\alpha}+\frac{\partial f_r}{\partial z_1^\beta}\frac{\partial^2g_\beta}{\partial z_2^\lambda\partial z_2^\alpha}v_2^\lambda,\label{bfc}
\end{eqnarray}
which yields $\breve{\varGamma}_{;\alpha}^\gamma v_2^\alpha=\hat{\varGamma}_{;\lambda}^\sigma v_1^\lambda\frac{\partial f_\gamma}{\partial z_1^\sigma}+\frac{\partial f_r}{\partial z_1^\beta}\frac{\partial^2g_\beta}{\partial z_2^\lambda\partial z_2^\alpha}v_2^\lambda v_2^\alpha$. This implies that $F_2$ is a weakly complex Berwald metric iff $F_1$ is a weakly complex Berwald metric.
By \eqref{bfc}, we also get
\begin{eqnarray}
\breve{\varGamma}_{\beta;\alpha}^{\gamma}=\hat{\varGamma}_{\mu;\lambda}^{\sigma}\frac{\partial f_\gamma}{\partial z_1^\sigma}\frac{\partial g_\lambda}{\partial z_2^\alpha}\frac{\partial g_\mu}{\partial z_2^\beta}+\frac{\partial f_\gamma}{\partial z_1^\lambda}\frac{\partial^2g_\lambda}{\partial z_2^\beta\partial z_2^\alpha}.\label{cc}
\end{eqnarray}
Note that \eqref{cc} imply that $F_2$ is a complex Berwald metric iff $F_1$ is a complex Berwald metric. This completes the assertion (3).

By \eqref{cc}  we also  have
\begin{equation}
\breve{\varGamma}_{\beta;\alpha}^\gamma-\breve{\varGamma}_{\alpha;\beta}^\gamma= (\hat{\varGamma}_{\mu;\lambda}^\sigma-\hat{\varGamma}_{\lambda;\mu}^\sigma)\frac{\partial f_\gamma}{\partial z_1^\sigma}\frac{\partial g_\lambda}{\partial z_2^\alpha}\frac{\partial g_\mu}{\partial z_2^\beta}\label{bkf}
\end{equation}
and
\begin{equation}
(G_2)_\gamma(\breve{\varGamma}_{\beta;\alpha}^\gamma-\breve{\varGamma}_{\alpha;\beta}^\gamma)v_2^\beta=\Big\{(G_1)_\nu (\hat{\varGamma}_{\mu;\lambda}^\sigma-\hat{\varGamma}_{\lambda;\mu}^\sigma)v_1^\mu\Big\}\frac{\partial f_\gamma}{\partial z_1^\sigma}\frac{\partial g_\lambda}{\partial z_2^\alpha}\frac{\partial f_\nu}{\partial z_2^\gamma},\label{bwkf}
\end{equation}
where we denote $(G_2)_\gamma=\frac{\partial G_2}{\partial v_2^\gamma}$  and $(G_1)_\nu=\frac{\partial G_1}{\partial v_1^\nu}$ and use the equality $v_1^\mu=\frac{\partial g_\mu}{\partial z_2^\beta}v_2^\beta$. Note that \eqref{bkf} and \eqref{bwkf} imply that $F_2$ is a K\"ahler-Finsler metric iff $F_1$ is a K\"ahler-Finsler metric, and $F_2$ is a weakly K\"ahler-Finsler metric iff $F_1$ is a weakly K\"ahler-Finsler metric. This complete the proof.
 \end{proof}

\begin{theorem}\label{bhm}
Let $M_1$ and $M_2$ be two complex manifolds and $M_1$ is biholomorphically equivalent to $M_2$. Then $M_2$ admits an $\mbox{Aut}(M_2)$-invariant strongly pseudoconvex complex Finsler metric iff $M_1$ admits an $\mbox{Aut}(M_1)$-invariant strongly pseudoconvex complex Finsler metric. More precisely, the following assertions hold:

(1) $M_2$ admits an $\mbox{Aut}(M_2)$-invariant Hermitian metric iff $M_1$ admits an $\mbox{Aut}(M_1)$-invariant Hermitian metric;

(2) $M_2$ admits an non-Hermitian quadratic $\mbox{Aut}(M_2)$-invariant strongly pseudoconvex complex Finsler metric iff $M_1$ admits an non-Hermitian quadratic $\mbox{Aut}(M_1)$-invariant strongly pseudoconvex complex Finsler metric;

(3) $M_2$ admits an $\mbox{Aut}(M_2)$-invariant complex Berwald metric (resp. weakly complex Berwald metric) iff $M_1$ admits an $\mbox{Aut}(M_1)$-invariant complex Berwald metric (resp. weakly complex Berwald metric);

(4) $M_2$ admits an $\mbox{Aut}(M_2)$-invariant K\"ahler-Finsler metric (resp. weakly K\"ahler-Finsler metric)  iff $M_1$ admits an $\mbox{Aut}(M_1)$-invariant K\"ahler-Finsler metric (resp. weakly K\"ahler-Finsler metric).
\end{theorem}

\begin{proof}
Note that since the strongly pseudoconvexity of a complex Finsler metric $F$, the notions of complex Berwald metric (resp. weakly complex Berwald metric), K\"ahler-Finsler metric (resp. weakly K\"ahler-Finsler metric) are all independent of the choice of local holomorphic coordinates on a complex manifold $M$. Thus by an argument of using local holomorphic coordinates $z=(z^1,\cdots,z^n)$ on $M$ and the induced local holomorphic coordinates $(z,v)$ on $T^{1,0}M$, and using Theorem \ref{bh}, we immediate obtain the assertions.
\end{proof}

\subsection{Some applications}

In this section we shall give some applications of Theorem \ref{mtha}, Theorem \ref{mthb} and Theorem \ref{bhm}.
\begin{proposition}
Let $D$ be a bounded Reinhardt domain in $\mathbb{C}^N$. Suppose that there exists a compact subset $K$ of $D$
such that Aut$(D)\cdot K=D$. Then $D$ admits infinite many non-Hermitian quadratic strongly convex K\"ahler-Berwald metrics with holomorphic sectional curvatures bounded from above by a negative constant.
\end{proposition}
\begin{proof}
By assumption and the main Theorem in \cite{Kodama}, $D$ is biholomorphically equivalent to a finite product of unit open balls, namely $D\cong B_1\times \cdots\times B_n$ with
$B_l\subset \mathbb{C}^{n_l}$ for $l=1,\cdots,n$. Thus the assertion follows immediately from Theorem \ref{mtha}, Theorem \ref{mthb}, Proposition \ref{prb} and Theorem \ref{bh}.

\end{proof}

\begin{proposition}
Let $M$ be a connected Stein manifold of dimension $n\geq 2$ and $\pmb{B}$ a domain in $\mathbb{C}^n$ given as the direct product $\pmb{B}=B_1\times\cdots\times B_{n_s}$ of balls, where each $B_{n_j}$ is the unit ball in $\mathbb{C}^{n_j}$ with $n_j>1$ and $\displaystyle\sum_{j=1}^sn_j=n$. Assume that there exists a topological subgroup $G$ of $\mbox{Aut}(M)$ that is isomorphism to $\mbox{Aut}(\pmb{B})$ as topological groups. Then $M$ admits infinite many strongly convex non-Hermitian quadratic K\"ahler-Berwald metrics with holomorphic sectional curvatures bounded from above by a negative constant.
\end{proposition}
\begin{proof}
By assumption and the main Theorem in \cite{KS}, it follows that $M$ is biholomorphically equivalent to $\pmb{B}$. Thus the assertion follows immediately from Theorem \ref{mtha}, Theorem \ref{mthb},  Proposition \ref{prb} and Theorem \ref{bh}.

\end{proof}

Note that both the unit ball $B_n$ and the unit polydisk are symmetric domains. But they admit different amounts of symmetry which are reflected in the size of the isotropy groups at the origin (cf. \cite{Hundemer}):

\begin{eqnarray*}
\mbox{Aut}_0(B_n)&=&\{z\mapsto Az: A\in U(n)\},\\
\mbox{Aut}_0^\circ(P_n)&=&\Big\{z\mapsto(e^{i\theta_1}z_1,\cdots,e^{i\theta_n}z_n)\Big| \theta_1,\cdots,\theta_n\in\mathbb{R}\Big\}\cong T^n=\underbrace{S^1\times\cdots\times S^1}_{\mbox{$n$\, copies}},\\
\mbox{Aut}_0(P_n)&=&\Big\{z\mapsto (e^{i\theta_{\sigma(1)}}z_{\sigma(1)},\cdots, e^{i\theta_{\sigma(n)}}z_{\sigma(n)})\Big| \theta_1,\cdots,\theta_n\in\mathbb{R}\;\mbox{and}\; \sigma\; \mbox{is a permutation}\Big\}.
\end{eqnarray*}

Let $\pmb{B}=B_1\times\cdots\times B_s$ be a product of unit balls $B_{i}\subset\mathbb{C}^{m_i}$ for $i=1,\cdots,s$. Then by Proposition \ref{C}, we have
$$\mbox{Aut}_p^\circ(\pmb{B})\cong\mbox{Aut}_{p_1}^\circ(B_1)\times\cdots\times \mbox{Aut}_{p_s}^\circ(B_s),$$
where $p=(p_1,\cdots,p_s)\in \pmb{B}$ such that $p_i\in B_i$ for $i=1,\cdots,s$.
Thus at the origin $0\in \pmb{B}$, we have
$$
\mbox{Aut}_0^\circ(\pmb{B})=U(m_1)\times \cdots\times U(m_s).
$$
Note that
$$T^n\subset \mbox{Aut}_0(B_n)\subset\mbox{Aut}(B_n)\quad \mbox{and} \quad T^n\subset\mbox{Aut}_0(P_n)\subset\mbox{Aut}(P_n).$$
 Since the polydisk and products of balls are homogeneous domains, thus  $\mbox{Aut}_p(D)$ contains a torus of dimension $n$ for each $p\in D$ if $D$ is biholomorphically equivalent to a product of balls.

A Siegel domain of the first kind over a regular cone $C\subset\mathbb{R}^n$ is the tube domain $\{z\in\mathbb{C}^n: \mbox{Im}\,z\in C\}$. Let $H:\mathbb{C}^k\times \mathbb{C}^k\rightarrow\mathbb{C}^n$ be a $C$-Hermitian form satisfying
(i) $H$ is $\mathbb{C}$-linear in the first argument; (ii) $H(z,w)=\overline{H(w,z)}$; (iii) $H(z,z)\in\overline{C}$; (iv) $H(z,z)=0$ iff $z=0$.
The Siegel domain of the second kind over $C$ with $C$-Hermitian form $H:\mathbb{C}^k\times\mathbb{C}^k\rightarrow \mathbb{C}^n$ is defined to be
$$
\{(z,w)\in\mathbb{C}^{n+k}:\mbox{Im}\,z-H(w,w)\in C\}.
$$
The pair $(n,k)$ is called the type of the Siegel domain, which is a biholomorphic invariant \cite{Hundemer}.

Let $D\subset\mathbb{C}^n$ be a domain, we use $\cong$ to denote two domains which are biholomorphically equivalent. Setting
\begin{eqnarray*}
\textbf{BSD}_n&=&\{D\;\cong\mbox{a bounded symmetric domain}\},\\
\textbf{BHD}_n&=&\{D\;\cong\mbox{a bounded homogeneous domain}\},\\
\textbf{SD1}_n&=&\{D\;\cong\mbox{a Siegel domain of the first kind}\},\\
\textbf{SD2}_n&=&\{D\;\cong\mbox{a Siegel domain of the second kind}\},\\
\textbf{BNC}_n&=&\{D\;\cong\mbox{a  bounded domain with a noncompact}\,\mbox{Aut}(D)\}.
\end{eqnarray*}

Then the following inclusions hold (cf. \cite{Hundemer})
$$
\textbf{SD1}_n\subset \textbf{SD2}_n\quad \mbox{and}\quad \textbf{BSD}_n\subset\textbf{BHD}_n\subset\textbf{SD2}_n\subset\textbf{BNC}_n.
$$
Furthermore, $\textbf{BSD}_n\neq \textbf{BHD}_n$ for every $n>3$ (a well-known result of I.I. Pyatetskii-Shapiro), $\textbf{BHD}_n\neq \textbf{SD2}_n$ for $n>2$, $\textbf{SD1}_n\neq \textbf{SD2}_n$ for $n>1$ and $\textbf{SD2}_n\neq \textbf{BNC}_n$ for all $n$.

 \begin{proposition}
 Let $D$ be a doamin in $\textbf{SD1}_n$. If there is a point $p\in D$ such that $\mbox{Aut}_pD$ contains a torus of dimension $n$, then $D$ admits an $\mbox{Aut}(D)$-invariant non-Hermitian quadratic strongly convex K\"ahler-Berwald metric with holomorphic sectional curvature bounded above by a negative constant.
\end{proposition}
\begin{proof}
By assumption and Proposition 3.5 in \cite{Hundemer}, it follows that $D$ is biholomorphically equivalent to the polydisk $P_n$. Thus the assertion follows from Theorem \ref{mtha}, Theorem \ref{mthb}, Theorem \ref{phi} and Theorem \ref{bh}.

\end{proof}

\begin{proposition}
Let $D$ be a domain in $\textbf{SD2}_n$.

(1) If there is a point $p\in D$ such that $\mbox{Aut}_p(D)$ contains a torus of dimension $n$ and
$$\dim\{\mbox{center}\; \mbox{Aut}_p^\circ(D)\}=1,$$
 then $D$ admits no $\mbox{Aut}(D)$-invariant non-Hermitian quadratic strongly pseudoconvex complex Finsler metric;

(2) If there is a point $p\in D$ such that
$$\dim\{\mbox{center}\; Aut_p^\circ(D)\}=n>1,$$
 then $D$ admits infinite
many  $\mbox{Aut}(D)$-invariant non-Hermitian quadratic strongly convex K\"ahler-Berwald metrics with holomorphic sectional curvatures bounded above by a negative constant.
\end{proposition}

\begin{proof}
(1) By assumption and Corollary 4.4 in \cite{Hundemer}, $D$ is biholomorphically equivalent to the unit ball $B_n$. Thus by Theorem \ref{nb} and Theorem \ref{bh}, the assertion follows.

(2) By assumption and Corollary 4.4 in \cite{Hundemer}, $D$ is biholomorphically equivalent to the polydisk $P_n$. Thus by Theorem \ref{phi} and Theorem \ref{bh}, the assertion follows.

\end{proof}

\begin{proposition}
Let $D$ be a domain in $\textbf{SD2}_n$ of type $(m,k)$.

(1) If $m=1$ and there is a point $p\in D$ such that $\mbox{Aut}_p(D)$ contains a torus of dimension $n$, then $D$ admits no $\mbox{Aut}(D)$-invariant non-Hermitian quadratic strongly pseudoconvex complex Finsler metric;

(2) If $m=n>1$ and there is a point $p\in D$ such that $\mbox{Aut}_p(D)$ contains a torus of dimension $n$, then $D$ admits infinite
many $\mbox{Aut}(D)$-invariant non-Hermitian quadratic strongly convex K\"ahler-Berwald metrics with holomorphic sectional curvatures bounded above by a negative constant.

\end{proposition}

\begin{proof}
(1) By assumption and Corollary 4.6 in \cite{Hundemer}, $D$ is biholomorphically equivalent to the unit ball $B_n$. Thus by Theorem \ref{nb} and Theorem \ref{bh}, the assertion follows.

(2) By assumption and Corollary 4.6 in \cite{Hundemer}, $D$ is biholomorphically equivalent to the polydisk $P_n$. Thus by Theorem \ref{phi} and Theorem \ref{bh}, the assertion follows.

\end{proof}

In \cite{Hundemer}, Hundemer proved that a Siegel domain $D\subset\mathbb{C}^n$ of the second kind is biholomorphically equivalent to a product of balls iff there is a point $p\in D$ such that the isotropy group of $p$ contains a torus of dimension $n$, and proved that the only domains biholomorphically equivalent to a Siegel domain of the second kind and to a Reinhardt domain are exactly the domains biholomorphically equivalent to a product of balls.

A complex manifold $M$ is called Kobayashi-hyperbolic if the Kobayashi pseudometric $F_M$ is a metric, which in return induces an $\mbox{Aut}(M)$-invariant pseudodistance, namely the Kobayashi pseudodistance. For a Kobayashi-hyperbolic manifold $M$ of complex dimension $n$, it is known (cf. \cite{Ka}) that
$$\dim\mbox{Aut}(M)\leq n^2+n,$$
 and $\mbox{Aut}(M)$ is a Lie group with respect to the compact-open topology.

\begin{theorem}
Let $M$ be a connected Kobayashi-hyperbolic manifold of complex dimension $n$. Suppose that $\dim\mbox{Aut}(M)=n^2+2n$. Then $M$ admits no $\mbox{Aut}(M)$-invariant strongly pseudoconvex complex Finsler metric other than a constant multiple of the Bergman metric, namely $M$ admits no $\mbox{Aut}(M)$-invariant non-Hermitian quadratic strongly pseudoconvex complex Finsler metric.
\end{theorem}

\begin{proof}
By our assumption and Theorem 1.1 in \cite{Is} (cf. also \cite{Ka}, \cite{Ko}), $M$ is biholomorphically equivalent to $B_n$, thus by Theorem \ref{nb} and Theorem \ref{bhm} the assertion follows.
\end{proof}

Using Theorem \ref{mtha} and Theorem \ref{mthb}, we obtain the following  example of non-Hermitian quadratic metric with positive holomorphic sectional curvature.
\begin{proposition}
Let $M=\mathbb{CP}^{m_1}\times \cdots \times \mathbb{CP}^{m_n}$  be the product of complex projective spaces $\mathbb{CP}^{m_l}$
endowed with the Fubini-Study metric
\begin{eqnarray*}
Q_l(z_l,v_l)=\frac{\displaystyle\sum_{\alpha_l=1}^{m_1}|v_l^{\alpha_l}|^2}{1+\Big(\displaystyle\sum_{\alpha_l=1}^{m_l}|z_l^{\alpha_l}|^2\Big)}+\frac{\displaystyle\sum_{\alpha_l=1}^{m_l}z_l^{\alpha_l}\overline{v_l^{\alpha_l}}}{\Big[1+\Big(\displaystyle\sum_{\alpha_l=1}^{m_l}|z_l^{\alpha_l}|^2\Big)\Big]^2},\quad v_l\in T_{z_l}^{1,0}\mathbb{CP}^{m_l}
\end{eqnarray*}
for $l=1,\cdots,n$. Then for any fixed $t\in[0,+\infty)$ and integer $k\geq 2$,
\begin{eqnarray*}
F_{t,k}(z,v)&=&\frac{1}{\sqrt{1+t}}\sqrt{\sum_{l=1}^nQ_l(z_l,v_l)+t\sqrt[k]{\sum_{l=1}^nQ_l^k(z_l,v_l)}},\quad \forall (z,v)\in T^{1,0}M
\end{eqnarray*}
is a strongly convex K\"ahler-Berwald metrics on $M$ with holomorphic sectional curvatures $\in\left[\frac{4(1+t)}{n+t\sqrt[k]{n}},4\right]$.
\end{proposition}
\begin{proof}
Note that by assumption, $Q_l$ is a K\"ahler metric on  $M_l=\mathbb{CP}^{m_l}$ with holomorphic sectional curvature $4$ for $l=1,\cdots,n$ and $M$ is the product manifold of $M_1,\cdots,M_n$. Thus by Theorem \ref{mtha}, Theorem \ref{mthb}, the assertion follows.
\end{proof}

The following theorem is a refinement of the Theorem 5.25 in \cite{Deng}.

\begin{theorem}\label{thm-hsm}
Let $M_l=G_l/H_l$ be Hermitian symmetric spaces endowed with $G_l$-invariant Hermitian metrics $Q_l$ for $l=1,\cdots,n$ and $M=(G_1/H_1)\times \cdots\times (G_n/H_n)$ the product manifold. For any $t\in[0,+\infty)$ and integer $k\geq 2$, define
\begin{equation}
F_{t,k}(z,v)=\frac{1}{\sqrt{1+t}}\sqrt{\sum_{l=1}^nQ_l(\pmb{\pi}_l(z),(\pmb{\pi}_l)_\ast(v))+t\sqrt[k]{\sum_{l=1}^nQ_l^k(\pmb{\pi}_l(z),(\pmb{\pi}_l)_\ast(v))}},\quad \forall (z,v)\in T^{1,0}M.\label{hsm}
\end{equation}
 Then

(1) $F_{t,k}$ is a strongly convex K\"ahler-Berwald metric $F_{t,k}$ on $M$;

(2) $F_{t,k}$ is invariant under $G_1\times\cdots\times G_n$ and makes $(M,F_{t,k})$ a symmetric complex Finsler space;

(3) if $K_l\equiv c\geq 0$, then $K_{t,k}\in\left[\frac{(1+t)c}{n+t\sqrt[k]{n}},c\right]$;

(4) if $K_l\equiv c<0$, then $K_{t,k}\in \left[c,\frac{(1+t)c}{n+t\sqrt[k]{n}}\right]$.
\end{theorem}

\begin{proof}
Let $o_l$ be the origin of $G_l/H_l$ for $l=1,\cdots,n$ and $o=(o_1,\cdots,o_n)$ the origin of the product manifold $M=(G_1/H_1)\times \cdots\times (G_n/H_n)$.
Let $z=(z_1,\cdots, z_n)\in M,v=(v_1,\cdots,v_n)\in T_z^{1,0}M$ such that $\pmb{\pi}_l(z)=z_l=(z_l^1,\cdots, z_l^{m_l})\in G_l/H_l$ and  $(\pmb{\pi}_l)_\ast(v)=v_l=(v_l^1,\cdots, v_l^{m_l})\in T_{z_l}^{1,0}(G_l/H_l)$ for $l=1,\cdots,n$. Then $F_{t,k}$ defined by \eqref{hsm} is invariant under $H_1\times \cdots\times H_n$, hence by Theorem \ref{mtha} and Theorem \ref{mthb}, the assertion follows.
\end{proof}

\section{de Rahm's decomposition theory of strongly convex K\"ahler-Berwald spaces}

It is well-known that a connected, simply connected and complete Riemannian manifold $M$ is isometric to the direct product $M_0\times M_1\times \cdots\times M_n$, where $M_0$ is a Euclidean space (possibly of dimension $0$) and $M_1, \cdots, M_n$ are all simply connected, complete, irreducible Riemannian manifolds. Such a decomposition is unique up to an order. This theorem is called the de Rham decomposition theorem which is due to G. de Rham \cite{Rham} (cf. also Theorem 6.2 in Chapter IV of volume I \cite{KN}). Few years after de Rham's result, Hano and Matsushima obtained following theorem (cf. also Theorem 8.1 in Chapter IX of volume II \cite{KN}).
 \begin{theorem}[\cite{HM}]\label{kd}
 A simply connected, complete K\"ahler manifold $M$ is holomorphic isometric to the direct product $M_0\times M_1\times \cdots\times M_k$, where $M_0$ is a complex Euclidean space and $M_1,\cdots,M_k$ are all simply connected, complete, irreducible K\"ahler manifolds.
 \end{theorem}

 The de Rham decomposition theorem for Riemannian manifolds was generalized to Finsler spaces (manifolds endowed with metrics which are not necessary quadratic). In \cite{Szabo}, Szab\'{o} proved that besides the Riemannian and locally real Minkowski spaces, there are $54$ non-Riemannian real Berwald spaces such that all other real Berwald spaces can be represented locally as a Descartes product of these spaces, thus obtained a de Rham type theorem for real Berwald spaces.





\begin{remark}
A strongly convex complex Berwald metric is not necessary a real Berwald metric. For example, for each strongly convex complex Minkowski
metric $f(v)$ on $\mathbb{C}^n$, the conformal change $F(z,v):=e^{\sigma(z)}f(v)$ of $f(v)$ is a strongly convex complex Berwald metric on $\mathbb{C}^n$.
It is a real Berwald metric on $\mathbb{C}^n$ if only if $\sigma(z)=\mbox{constant}$. On the other hand, a real Berwald metric on
a complex manifold $M$ is not necessary reversible or compatible with the complex structure $J$ on $M$, thus not
necessary a complex Berwald metric.
\end{remark}

In section 3, it is proved that a reducible Hermitian manifold (resp. K\"ahler manifold) $(M,Q)$  always admits non-Hermitian complex Berwald metric (resp. K\"ahler-Berwald metric) $F_{t,k}$ such that
$(M,F_{t,k})$ is a non-Hermitian complex Berwald manifold (resp. K\"ahler-Berwald manifold). In this section, we show that each simply connected complete
 K\"ahler-Berwald manifold $(M,F)$ admits a de Rham type decomposition, that is $(M,F)$ is  actually a product of K\"ahler-Berwald manifolds $(M_1,F_1),\cdots,(M_n,F_n)$.

For this purpose, we introduce the notion of product of complex Berwald manifolds, which is a generalization of the usual product of Hermitian manifolds.
Suppose that $(M_l,F_l)$ are complex Berwald manifolds for $l=1,\cdots,n$. Then by Proposition 5.1 in \cite{Aikou1} or Proposition 4.3 in \cite{Aikou2}, there are Hermitian metrics $Q_l$ on $M_l$ for $l=1,\cdots,n$
such that the pull-back of the Hermitian connection $\nabla^{(l)}$ associated to $Q_l$ coincide with the horizontal Chern-Finsler connection associated to $F_l$
 for $l=1,\cdots,n$. Define a complex linear connection
$$\nabla:=\nabla^{(1)}\times\cdots\times \nabla^{(n)}$$
 on $M=M_1\times \cdots\times M_n$. By Theorem \ref{mtha}, there  exists a non-Hermitian quadratic complex Berwald metric $F_{t,k}$ on $M$
such the pull-back of $\nabla$ coincides with the horizontal Chern-Finsler connection associated to $F_{t,k}$, namely
\begin{eqnarray*}
\nabla_{(\pmb{\pi}_l)^\ast X}(\pmb{\pi}_l)^\ast Y=(\pmb{\pi}_l)^\ast(\nabla_X^{(l)}Y)\quad \mbox{and}\quad
\nabla_{(\pmb{\pi}_l)^\ast X}(\pmb{\pi}_i)^\ast Z=0,
\end{eqnarray*}
where $X$ and $Y$ are vector fields of type $(1,0)$ on $M_l$  and $Z$ is a vector field of type $(1,0)$ on $M_i$ with $l\neq i$.

\begin{definition}\label{pb}
A complex Berwald manifold $(M,F)$ is said to be the Descartes product of the complex Berwald manifolds $(M_1,F_1),\cdots,(M_n,F_n)$, if

(1) $M=M_1\times\cdots\times M_n$;

(2) $\nabla=\nabla^{(1)}\times\cdots\times \nabla^{(n)}$ is the horizontal Chern-Finsler connection associated to $F$ and $\nabla^{(l)}$ are the complex Berwald connections associated to $F_l$ for $l=1,\cdots,n$;

(3) $F(p,(\pmb{\pi}_l)^\ast X)=F_l(\pmb{\pi}_l(p), X)$ for any $p\in M, X\in T_{\pmb{\pi}_l(p)}^{1,0}(M_l)$ and $l=1,\cdots,n$.

\end{definition}
\begin{remark}
If $(M_l,Q_l),l=1,\cdots,n$ are  Hermitian manifolds and $F$ is the usual product metric of $Q_1,\cdots,Q_n$ on $M$,
then $(M,F)$ obviously satisfies Definition \ref{pb}. If $(M_l,F_l),l=1,\cdots,n$, are non-Hermitian complex Berwald manifolds, then by Proposition 5.1 in \cite{Aikou1} or Proposition 4.3 in \cite{Aikou2}, there are
Hermitian metrics $Q_l$ on $M_l$ such that the pull-back of the Hermitian connection associated to $Q_l$
coincides with the horizontal Chern-Finsler connection associated to $F_l$, taking $F=F_{t,k}$ as defined in Theorem \ref{mtha}, one can check that $(M,F_{t,k})$ also satisfies Definition \ref{pb}.
\end{remark}


\begin{theorem}\label{mth-c}
A connected, simply connected, strongly convex K\"ahler-Berwald space $(M,F)$ must be one of the following four types:

1) $(M,F)$ is a Hermitian space.

2) $(M,F)$ is a locally complex Minkowski space.

3) $(M,F)$ is a locally irreducible and locally symmetric non-Hermitian K\"ahler-Berwald space of rank $r\geq 2$.

4) $(M,F)$ is locally reducible, and in this case $(M,F)$ can be locally decomposed into a Descartes product of Hermitian spaces,
locally complex Minkowski spaces and locally irreducible symmetric non-Hermitian K\"ahler-Berwald spaces of rank $r\geq 2$.
\end{theorem}

\begin{proof}
On one hand, by our assumption and Proposition 4.4 in \cite{Aikou2}, $M$ is necessary a connected, simply connected K\"ahler manifold, namely there
exists a K\"ahler metric $Q$ on $M$. Thus by Theorem \ref{kd},
$(M,Q)$ is holomorphic isometric to the direct product $M_0\times M_1\times \cdots\times M_n$, where $M_0$ is a complex Euclidean space and $M_1,\cdots,M_k$ are all simply connected, irreducible K\"ahler manifolds. If furthermore $(M,Q)$ is a simply connected Hermitian symmetric space, then by Proposition 4.4 in \cite{Hel}, $M$ is a product $M_0\times M_{-}\times M_{+},$ where $M_0$ is a Euclidean Hermitian symmetric space, $M_{-}$ is a Hermitian symmetric space of non-compact type and $M_{+}$ is a Hermitian symmetric space of compact type. Moreover, $M_{-}$ (resp. $M_{+}$) is simply connected and it decomposes uniquely as a product of non-compact (resp. compact) irreducible Hermitian symmetric spaces.

On the other hand,  by our assumption $(M,F)$ and Theorem 1.2 and Theorem 1.4 in \cite{Zh1}, it follows  that $F$
is also a real Berwald metric on $M$, the
real Berwald connection and complex Berwald connection associated to $F$ coincide and are equal to the horizontal Chern-Finsler connection associated to $F$. Thus $(M,F)$ satisfies the assumption of Theorem 3 in \cite{Szabo}, and the assertions of Theorem 3 in \cite{Szabo} hold automatically for $(M,F)$, which implies the assertions of Theorem \ref{mth-c}.

\end{proof}

\begin{theorem}\label{mth-d}
A connected, simply connected, complete, strongly convex K\"ahler-Berwald space $(M,F)$
can be decomposed into the Descartes product of a complex Minkowski space,
simply connected complete irreducible K\"ahler spaces and simply connected complete irreducible globally
symmetric non-Hermitian K\"ahler-Berwald space of rank $r\geq 2$. Such a decomposition is unique up to an order.
\end{theorem}

\begin{proof}
By our assumption, Theorem 1.2 and Theorem 1.4 in \cite{Zh1}, it follows that the real Berwald connection and complex Berwald connection associated to $F$ coincide \cite{Zh1} (denoted by $\nabla$), and both of them are complete in the sense that every geodesic $z(t)$ of $F$ (or equivalently $\nabla$) can be extended to a geodesic defined for all $t\in(-\infty,+\infty)$. By our assumption and Proposition 4.4 in \cite{Aikou2}, there exists a complete K\"ahler metric $Q$ on $M$ which is $J$ invariant and the pull-back Hermitian connection associated to $Q$ coincides with the horizontal Chern-Finsler connection $\nabla$ associated to $F$, thus both of them have the same systems of geodesic equation. Therefore $(M,Q)$ is a connected, simply connected, complete K\"ahler manifold. Thus by Theorem \ref{kd}, $(M,Q)$ is holomorphic isometric to the direct product $M_0\times M_1\times\cdots\times M_k$, where $M_0$ is a complex Euclidean space and $M_i,i=1,\cdots,k$ are connected, simply connected, complete and irreducible K\"ahler manifolds. Now denote $(M_0, F_i,\nabla^{(i)})$ the restriction of the strongly convex K\"ahler-Berwald space $(M,F,\nabla)$ onto $M_i,i=0,1,\cdots,k$. Then $(M_0,F_0)$ is clear a complex Minkowski space, and for $i\geq 1$, $(M_i,F_i)$ is either an irreducible K\"ahler space, or an irreducible connected, complete, and globally symmetric non-Hermitian K\"ahler-Berwald space of rank $r\geq 2$. Thus by Theorem \ref{mth-c} and Definition \ref{pb}, $(M,F)$  is the Descartes product of these K\"ahler-Berwald spaces.
\end{proof}

\begin{remark}
The assumption that $F$ is a strongly convex  K\"ahler-Berwald metric is necessary for Theorem \ref{mth-c} and Theorem \ref{mth-d}, this has already been seen in Hermitian quadratic case (Theorem \ref{kd}).
\end{remark}
\vskip0.4cm
\noindent
\textbf{Acknowledgement.}\ {\small This work is partially supported by NSFC (Grant No. 12071386, 11671330).

\end{document}